\documentclass[a4paper,twoside,12pt]{article}
\usepackage[utf8]{inputenc}
\usepackage[russian,english]{babel}
\usepackage{graphicx}
\usepackage{color}
\usepackage{tikz}\usetikzlibrary{calc, patterns}
\usepackage{dsfont} 
\usepackage{mathrsfs} 
\usepackage{enumitem}
\usepackage[body={160mm,220mm},centering,nomarginpar]{geometry}
\usepackage{hyperref}
\usepackage[all]{hypcap}

\newcommand{\ignore}[1]{}

\hypersetup{
 pdftitle={Connected components of meanders: I. Bi-rainbow meanders},
 pdfauthor={Anna Karnauhova and Stefan Liebscher},
 colorlinks=true,
 linkcolor=blue,
 citecolor=blue,
 filecolor=magenta,
 urlcolor=magenta}

\makeatletter
\@addtoreset{equation}{section}
\@addtoreset{figure}{section}
\makeatother

\newtheorem{prop}{\bf Proposition}[section]
\newtheorem{defn}[prop]{\bf Definition}
\newtheorem{lem}[prop]{\bf Lemma}
\newtheorem{cor}[prop]{\bf Corollary}

\newtheorem{thm}[prop]{\bf Theorem}

\newenvironment{proof}
  {\begin{trivlist}\item[]{\bf Proof.}}
  {\hfill{\boldmath$\bowtie$}\end{trivlist}}

\newcommand{\setR}{\mathds{R}}
\newcommand{\setZ}{\mathds{Z}}
\newcommand{\setN}{\mathds{N}}

\newcommand{\components}{\ensuremath{\mathop{\mathcal{Z}\strut}}}
\newcommand{\pathComponents}{%
 \ensuremath{\mathop{\mathcal{Z}_\mathrm{path}\strut}}}
\newcommand{\cycleComponents}{%
 \ensuremath{\mathop{\mathcal{Z}_\mathrm{cycle}\strut}}}
\newcommand{\meander}{\ensuremath{\mathop{\mathcal{M}\strut}}}
\newcommand{\collapsedMeander}{\ensuremath{\mathop{\mathcal{CM}\strut}}}
\newcommand{\rainbowMeander}{\ensuremath{\mathop{\mathcal{RM}\strut}}}
\newcommand{\collapsedRainbowMeander}{\ensuremath{\mathop{\mathcal{CRM}\strut}}}
\newcommand{\arcs}{\ensuremath{\alpha}}
\newcommand{\bits}{\ensuremath{b}}
\newcommand{\arcsabove}{{
  \begin{tikzpicture}
    \path (0,-0.3ex) rectangle (1ex,0.3ex);
    \draw [thin] (0,0) -- (1ex,0) (0.2ex,0) arc (180:0:0.3ex);
  \end{tikzpicture}}}
\newcommand{\arcsbelow}{{
  \begin{tikzpicture}
    \path (0,-0.3ex) rectangle (1ex,0.3ex);
    \draw [thin] (0,0) -- (1ex,0) (0.2ex,0) arc (180:360:0.3ex);
  \end{tikzpicture}}}
\newcommand{\arcsboth}{{
  \begin{tikzpicture}
    \draw [thin] (0,0) -- (1ex,0) (0.5ex,0) circle (0.3ex);
  \end{tikzpicture}}}
\newcommand{\remainder}{\ensuremath{\mathop{\mathscr{R}\strut}}}

\newcommand{\Ord}{\ensuremath{\mathop{\mathcal{O}\strut}}}

\newcommand{\meanderArc}[2]{(#1,0) arc (0:180:0.5*#1-0.5*#2 and 0.25*#1-0.25*#2)}
\newcommand{\drawMeanderArc}[2]{\draw [thick] \meanderArc{#1}{#2};}
\newcommand{\drawMeanderRainbow}[3]{\foreach \slX in {2,4,...,#3} {
  \coordinate (tmpA) at (#1,0);
  \coordinate (tmpB) at (#2,0);
  \coordinate (tmpX) at (barycentric cs:tmpB=\slX/2-1,tmpA=#3-\slX/2);
  \coordinate (tmpM) at (barycentric cs:tmpA=0.5,tmpB=0.5);
  \draw [thick] let \p1 = ($(tmpX)-(tmpM)$) in
                (tmpX) arc (0:180:\x1 and \x1/2);}}
\newcommand{\meanderPath}[1]{
  \foreach \slX in {#1} {
    let \p1=({abs(\slX)},0), \p2=++(0,0) in
    arc (180:0:\x1/2-\x2/2 and {abs(\x1/4-\x2/4)*\slX/abs(\slX)}) }}
\newcommand{\drawMeanderPath}[2]{\draw [thick] (#1,0) \meanderPath{#2};}
\definecolor{background}{gray}{0.8}

\begin{document}
\pagestyle{empty}
\markboth
  {A.~Karnauhova and S.~Liebscher}
  {Connected components of meanders: I. Bi-rainbow meanders}
\pagenumbering{alph}

\null\vfill

\begin{center}
\Huge\bfseries
Connected components of meanders:

\LARGE\bfseries
I.~Bi-rainbow meanders
\end{center}

\vfill

\begin{center}
\textbf{\large Anna Karnauhova}

\texttt{anna.karnauhova@math.fu-berlin.de}

\bigskip

\textbf{\large Stefan Liebscher}

\texttt{stefan.liebscher@fu-berlin.de}

\vfill

   Freie Universität Berlin, Institut für Mathematik
\\ Arnimallee 3, 14195 Berlin, Germany
\end{center}

\vfill

\begin{center}
\includegraphics[height=0.4\vsize]{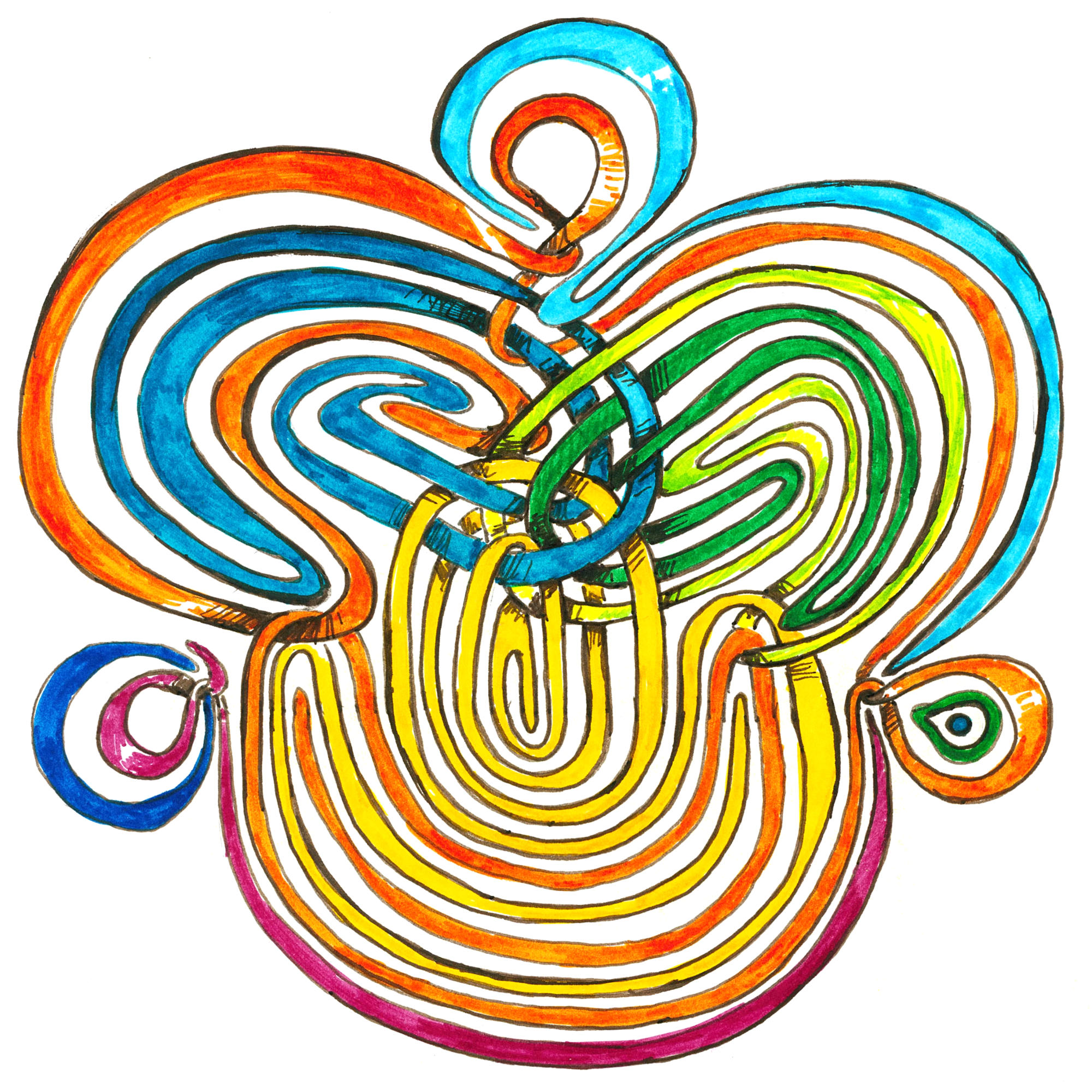}
\end{center}

\vfill

\begin{center}
draft version of April 13, 2015
\end{center}


\vfill\null

\clearpage

\null\vfill

\begin{abstract}

\noindent
\emph{Closed meanders} are planar configurations of
one or several disjoint closed Jordan curves intersecting
a given line or curve transversely.
They arise as shooting curves of parabolic PDEs in one space dimension,
as trajectories of Cartesian billiards,
and as representations of elements of Temperley-Lieb algebras.

Given the configuration of intersections, for example as a permutation
or an arc collection, the number of Jordan curves is unknown and needs
to be determined.

We address this question in the special case of \emph{bi-rainbow meanders},
which are given as non-branched families (rainbows) of nested arcs.
Easily obtainable results for small bi-rainbow meanders containing
up to four families suggest an expression of the number of curves by the
greatest common divisor (gcd) of polynomials in the sizes of the rainbow
families.

We prove however, that this is not the case.
In fact, the number of connected components of bi-rainbow meanders with
more than four families cannot be expressed as the gcd of
polynomials in the sizes of the rainbows.

On the other hand, we provide a complexity analysis of
\emph{nose-retraction} algorithms.
They determine the number of connected components of arbitrary
bi-rainbow meanders in logarithmic time.
In fact, the nose-retraction algorithms resemble the Euclidean algorithm,
which is used to determine the gcd, in structure and complexity.

Looking for a \emph{closed formula} of the number of connected components,
the nose-retraction algorithm is as good as a gcd-formula and therefore
as good as we can possibly expect.

\end{abstract}

\vfill\null

\cleardoublepage
\pagenumbering{arabic}
\setcounter{page}{1}
\pagestyle{myheadings}

\section{Introduction\label{secIntro}}

Meander curves in the plane emerge
as shooting curves of parabolic PDEs in one space dimension
\cite{FiedlerRocha1999-SturmPermutations},
as trajectories of Cartesian billiards
\cite{FiedlerCastaneda2012-Meander},
and as representations of elements of Temperley-Lieb algebras
\cite{DiFrancescoGolinelliGuitter1997-Meander}.

In general, we regard closed meanders as the pattern created by one or several
disjoint closed Jordan curves in the plane as they intersect a given line
transversely.
The pattern of intersections remains the same when we deform the curves into
collections of arcs with endpoints on the horizontal axis,
see figure \ref{figMeanderHomotopy}.

They induce a permutation on the set of intersection points with the
horizontal axis. Starting from the permutation,
the inverse problem raises two questions:
First, is a given permutation a meander permutation, i.e.~is it generated
by a meander? Second, if yes, is it generated by a single curve,
or more generally, of how many curves is the meander composed of?

We study the subclass of \emph{bi-rainbow meanders},
which are composed of several non-branched families of nested arcs
above and below the axis, see figure \ref{figRainbowMeander} and
definition \ref{defRainbowMeander}.
This subclass also appears naturally as representations of seaweed algebras
\cite{CollMagnantWang2012-Meander}.
For less than four upper families, the number of curves of the meander is
given as the greatest common divisor of expressions in the sizes of
the families,
\[
\begin{array}{lcl}
\components(\arcs_1) &=& \arcs_1, \\
\components(\arcs_1,\arcs_2) &=& \gcd(\arcs_1, \arcs_2), \\
\components(\arcs_1,\arcs_2,\arcs_3) &=& \gcd(\arcs_1+\arcs_2,\arcs_2+\arcs_3),
\end{array}
\]
see (\ref{eqGcdOfRainbows}) and \cite{FiedlerCastaneda2012-Meander}.
It is tempting to conjecture the existence of similar ``closed''
expressions for general bi-rainbow meanders. However, in section \ref{secNoGcd},
we prove that there are severe obstructions:

\begin{trivlist}
\item[\hskip \labelsep{\bfseries Theorem\ \ref{thmNoGcd}}]
\itshape
Let $n\ge4$ be given.
Then there do not exist homogeneous polynomials
$f_1,f_2 \in \setZ[x_1,\ldots,x_n]$ of arbitrary degree
with integer coefficients such that the
number of connected components $\components(\arcs_1,\ldots,\arcs_n)$
of every bi-rainbow meander $\rainbowMeander(\arcs_1,\ldots,\arcs_n)$
is given by the $\gcd(f_1(\arcs_1,\ldots,\arcs_n),f_2(\arcs_1,\ldots,\arcs_n))$.
In other words: to every choice of polynomials $f_1, f_2$,
we find a counterexample.
\end{trivlist}

After this partly negative result, we could look for more complicated formulae.
Instead of that, we shall shift our viewpoint a little bit.
We argue, that the $\gcd$ is an abbreviation for the Euclidean algorithm
rather than an ``explicit'' expression.
The Euclidean algorithm has logarithmic complexity: It requires
$\Ord(\log\arcs_1 + \log\arcs_2)$ steps to determine $\gcd(\arcs_1,\arcs_2)$.
Conversely, any formula for the number $\components$ of connected components
also provides a formula for the $\gcd$.

Therefore, whatever ``closed'' formula we find, it cannot be of smaller
complexity. In section \ref{secAlgorithms} we provide
algorithms which calculate the number of connected component
by nose retractions,
which we introduce in section \ref{secNoseRetractions}.
They have a structure very similar to the Euclidean algorithm.
They also have the same logarithmic complexity.
Although the search for exact $\gcd$ expressions might be futile,
we still find a $\gcd$-like interpretation of the number of connected
components of meanders.


\pdfbookmark[2]{Acknowledgements}{secAcknowledgements}
\subsection*{Acknowledgements}
Both authors have been supported by the Collaborative Research Centre 647
``Space--Time--Matter'' of the German Research Foundation (DFG).
We thank Bernold Fiedler, Pablo Casta\~neda, and Vincent Trageser
for useful discussions and encouragement.


\section{Meander curves\label{secMeanders}}

\begin{figure}
\centering
\begin{tikzpicture}[scale=0.75*\hsize/11cm, label/.style={ inner sep=2pt }]
\draw [thin] (0,0) -- (11,0);
\drawMeanderPath{1}{10,-7,8,-9,6,-1};
\drawMeanderPath{2}{5,-4,3,-2};
\node [label, above left ] at  (1,0)  {$1$};
\node [label, above left ] at  (2,0)  {$2$};
\node [label, above left ] at  (3,0)  {$3$};
\node [label, above right] at  (4,0)  {$4$};
\node [label, above right] at  (5,0)  {$5$};
\node [label, above left ] at  (6,0)  {$6$};
\node [label, above left ] at  (7,0)  {$7$};
\node [label, above right] at  (8,0)  {$8$};
\node [label, above right] at  (9,0)  {$9$};
\node [label, above right] at (10,0) {$10$};
\end{tikzpicture}
\caption{\label{figMeanderHomotopy}
A general meander as a pair of arc collections.}
\end{figure}
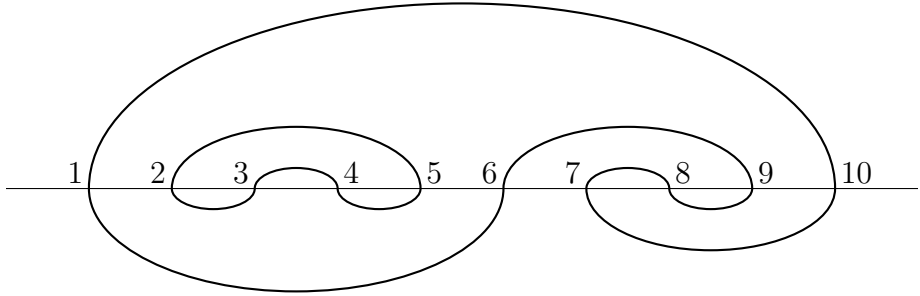

We start with one or several disjoint closed Jordan curves in the plane,
intersected by a given line --- without loss of generality the horizontal axis.
By homotopic deformations, without introduction or removal of intersections,
the curve can be represented by collections of arcs above and below the
horizontal axis. Both collections have the same number $\arcs$ of disjoint arcs
and hit the axis in the same $2\arcs$ points $\{1,\ldots,2\arcs\}$,
see figure \ref{figMeanderHomotopy}.
There are now several possibilities to represent a meander.

\subsection{...as a pair of products of disjoint
transpositions\label{secMeanderTranspositions}}

Each arc connects two points on the axis and thus defines a transposition.
The arc collections above and below the axis can both be represented as
products of the disjoint transpositions given by their arcs.
The example in figure \ref{figMeanderHomotopy} reads
\begin{equation}\label{eqTranspositionExample}
\begin{array}{rcl}
\pi^\arcsabove &=& (1,10)(2,5)(3,4)(6,9)(7,8), \\
\pi^\arcsbelow &=& (1,6)(2,3)(4,5)(7,10)(8,9),
\end{array}
\end{equation}
in the common cycle representation of permutations.
Such a product of disjoint transpositions represents a disjoint arc collection
if, and only if, no pair of transpositions is interlaced, i.e.
\[
\mbox{whenever} \quad a<b<\pi^{\arcsabove/\arcsbelow}(a) \quad \mbox{then} \quad
a<\pi^{\arcsabove/\arcsbelow}(b)<\pi^{\arcsabove/\arcsbelow}(a),
\qquad \pi^{\arcsabove/\arcsbelow} \in \{\pi^\arcsabove, \pi^\arcsbelow\}.
\]
Necessarily, both $\pi^\arcsabove$ and $\pi^\arcsbelow$
must interchange odd and even numbers,
\begin{equation}\label{eqTranspositionOddEven}
\begin{array}{rcl}
\pi^{\arcsabove/\arcsbelow}|_{\{\mathrm{odd}\}} = \{\mathrm{even}\},
\qquad
\pi^{\arcsabove/\arcsbelow}|_{\{\mathrm{even}\}} = \{\mathrm{odd}\}.
\end{array}
\end{equation}

\begin{prop}Let
\begin{equation}\label{eqCyclicPermutation}
\sigma \;:=\; (1,2,3,4,\ldots,2\arcs)
\end{equation}
be the cyclic permutation of the $2\arcs$ intersection points
with the axis.
Then a product $\pi^{\arcsabove/\arcsbelow}$
of $\arcs$ disjoint transpositions represents a disjoint arc collection if,
and only if,
the permutation $\pi^{\arcsabove/\arcsbelow} \sigma$
has exactly $\arcs+1$ (disjoint) cycles.
\end{prop}
\begin{proof}
Start with a disjoint arc collection.
Take the graph with the $v:=2\arcs$ vertices
$\{1,\ldots,2\arcs\}$.
The $e:=3\arcs-1$ edges are given by the $\arcs$ arcs of the collection together
with the $2\arcs-1$ edges
$\{(1,2),\ldots,(2\arcs-1,2\arcs)\}$ on the axis.
Each cycle of $\pi^{\arcsabove/\arcsbelow} \sigma$ corresponds to the
oriented boundary of a face.
By Euler's formula, the number of faces $f$ of a (planar) graph is given
by $f=e-v+2$. Therefore, there must be exactly $\arcs+1$ cycles.

Now start with an arbitrary arc collection.
If $\pi^{\arcsabove/\arcsbelow} \sigma$ has a fixed point $\ell$
then this fixed point belongs to an arc connecting the neighbouring points
$\ell, \ell+1$.
(The points $2\arcs$ and $1$ are also neighbours in this sense.)
This arc can be removed, decreasing $\arcs$ and the number of cycles by one
and keeping the structure of the remaining arcs (including their intersections)
and of the remaining cycles of $\pi^{\arcsabove/\arcsbelow} \sigma$.
If $\pi^{\arcsabove/\arcsbelow} \sigma$ does not have a fixed point
then all cycles have length at least 2.
Therefore there can be at most $\arcs$ cycles,
and the arc collection is not disjoint due to the first argument.

For disjoint initial arc collections $\pi^{\arcsabove/\arcsbelow}$,
the iterative removal of fixed points of $\pi^{\arcsabove/\arcsbelow} \sigma$
(alias arcs between neighbours) finishes at the trivial arc collection
$\tilde{\pi}^{\arcsabove/\arcsbelow} = (12)$ of a single arc.
Indeed, $\tilde{\pi}^{\arcsabove/\arcsbelow} \sigma = (1)(2)$
has two cycles and therefore
$\pi^{\arcsabove/\arcsbelow} \sigma$ has $\arcs+1$ cycles.

For non-disjoint initial arc collections $\pi^{\arcsabove/\arcsbelow}$,
the iterative removal of fixed points of $\pi^{\arcsabove/\arcsbelow} \sigma$
must stop earlier:
at an arc collection $\tilde{\pi}^{\arcsabove/\arcsbelow}$ such that
$\tilde{\pi}^{\arcsabove/\arcsbelow} \sigma$ has no fixed points.
The number of cycles of $\pi^{\arcsabove/\arcsbelow} \sigma$ is therefore
less than or equal to $\arcs$.
\end{proof}

\subsection{...as a single meander permutation\label{secMeanderPermutation}}

The two products of transpositions which we used in the previous section
both interchange odd and even numbers, see (\ref{eqTranspositionOddEven}).
Therefore, we can combine them into a single permutation
\begin{equation}\label{eqMenaderPermutation}
\pi^\arcsboth(k) \;=\;\left\{\begin{array}{ll}
   \pi^\arcsabove(k) &,\quad k \mbox{ odd},\\
   \pi^\arcsbelow(k) &,\quad k \mbox{ even}.
\end{array}\right.
\end{equation}
The cycles of this permutation directly correspond to the closed curves of
the meander. We find:

\begin{prop}
The number of connected components, i.e.~the number of closed Jordan curves,
of a meander equals the number of cycles of the associated
meander permutation $\pi^\arcsboth$.
Additionally, the number of cycles of $\pi^\arcsabove\pi^\arcsbelow$
is twice the number of connected components.
\end{prop}

The second part follows from the fact that
$\pi^\arcsabove\pi^\arcsbelow$
leaves the sets of even/odd numbers invariant and
equals $(\pi^\arcsboth)^2$ on the even and its inverse on the odd numbers.
For the example in figure \ref{figMeanderHomotopy}, we get
\begin{equation}\label{eqPermutationExample}
\begin{array}{rcl}
\pi^\arcsboth &=& (1,10,7,8,9,6)(2,3,4,5), \\
\pi^\arcsabove\pi^\arcsbelow &=& (1,9,7)(3,5)(2,4)(6,10,8).
\end{array}
\end{equation}

\subsection{...as a shooting permutation\label{secMeanderShooting}}

\begin{figure}
\centering
\begin{tikzpicture}[scale=0.75*\hsize/15cm, label/.style={ inner sep=2pt }]
\draw [thin] (0,0) -- (15,0);
\begin{scope}[shift={(0,0.5)}]
  \drawMeanderArc{12}{ 1}
  \drawMeanderArc{ 5}{ 2}
  \drawMeanderArc{ 4}{ 3}
  \drawMeanderArc{11}{ 6}
  \drawMeanderArc{10}{ 7}
  \drawMeanderArc{ 9}{ 8}
  \drawMeanderArc{14}{13}
\end{scope}
\begin{scope}[shift={(0,-0.5)}]
  \drawMeanderArc{ 1}{14}
  \drawMeanderArc{ 2}{ 3}
  \drawMeanderArc{ 4}{13}
  \drawMeanderArc{ 5}{ 6}
  \drawMeanderArc{ 7}{12}
  \drawMeanderArc{ 8}{11}
  \drawMeanderArc{ 9}{10}
\end{scope}
\foreach \slX/\slS in
  {1/1,2/10,3/11,4/12,5/9,6/8,7/3,8/6,9/5,10/4,11/7,12/2,13/13,14/14} {
  \draw [thick] (\slX,-0.5) -- (\slX,0.5);
  \node [label, above left ] at  (\slX,0)  {$\slX$};
  \node [label, below left ] at  (\slX,0)  {$\slS$};
}
\end{tikzpicture}
\caption{\label{figShootingCurve}
Shooting permutation of a connected meander.}
\end{figure}
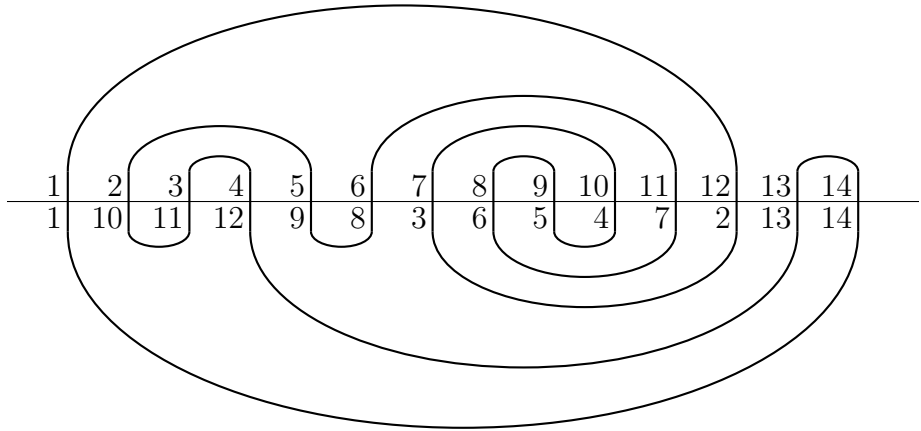

In \cite{Rocha1991-SturmAttractor, FiedlerRocha1999-SturmPermutations}
meander curves are found
as shooting curves of scalar reaction-advection-diffusion equations,
\begin{equation}\label{eqSturmPDE}
u_t \;=\; u_{xx} + f(x,u,u_x),
\end{equation}
in one space dimension, $x\in[0,L]$.
They are used to describe the global attractors of these systems.

Indeed, the $u$-axis $\{u_x|_{x=0}=0\}$, corresponding to a Neumann boundary
condition on the left boundary, is propagated by
\[
0 \;=\; u_{xx} + f(x,u,u_x)
\]
to a curve in the $(u,u_x)$-plane at the right boundary $x=L$.
In particular, intersections of this curve with the horizontal axis
yield stationary solutions of (\ref{eqSturmPDE})
with Neumann boundary conditions.

To facilitate its application in this context, the meander is described by a
permutation $\pi$ such that $(k,\pi(k))$ are the right and left boundary
values of the stationary solutions of the PDE.
In fact, the meander is connected by construction,
i.e.~consists of only one Jordan curve.
It is originally open, going to $\pm\infty$ for large $|u|$,
but can be artificially closed.
The permutation used in \cite{Rocha1991-SturmAttractor} yields
\begin{equation}\label{eqMenaderShootingPermutation}
\pi\left( (\pi^\arcsboth)^{k}(1) \right) \;=\; k+1,
\qquad
\pi^\arcsboth \;=\; \pi^{-1} \sigma \pi,
\end{equation}
with the cyclic permutation $\sigma$ as in (\ref{eqCyclicPermutation}).
The shooting permutation $\pi$ maps the enumeration of intersections along
the horizontal axis onto the enumeration along the shooting curve,
see figure \ref{figShootingCurve}
for an illustration of an artificially closed shooting curve
and \cite{FiedlerRocha1999-SturmPermutations, FiedlerRocha2009-SturmAttractors1}
for recent results on attractors of (\ref{eqSturmPDE}).

\subsection{...as a (condensed) bracket expression\label{secMeanderBracket}}

When we replace each arc by a pair of brackets, with the opening bracket
at the left end and the closing bracket at the right end of the arc,
we find corresponding balanced bracket expressions.
The example in figure \ref{figMeanderHomotopy} is represented by
\begin{equation}\label{eqBracketExample}
\frac{\mathds{\bigg[\,\Big[\;[\;]\;\Big]\,\Big[\;[\;]\;\Big]\,\bigg]}}
  {\Big[\;[\;]\;[\;]\;\Big]\,\Big[\;[\;]\;\Big]}.
\end{equation}
Each bracket expression can be condensed to a tuple of pairs of
positive integers
\begin{equation}\label{eqCondensedBracketExpression}
 ((\arcs_1^[,\arcs_1^]),\ldots,(\arcs_n^[,\arcs_n^])),
\end{equation}
with $\arcs_k^[$ representing consecutive opening brackets alias
left ends of arcs and
$\arcs_k^]$ representing consecutive closing brackets alias right ends of arcs.
Zero entries could be allowed but can always be removed.
The example in figure \ref{figMeanderHomotopy} then reads
\begin{equation}\label{eqCondensedBracketExample}
\frac{((3,2),(2,3))}{((2,1),(1,2),(2,2))}.
\end{equation}

On the other hand, each bracket expression represents a disjoint arc
collection provided
\begin{equation}\label{eqCondensedBracketExpressionConditions}
 \forall k \;\; \sum_{\ell=1}^k \arcs_\ell^[ \ge \sum_{\ell=1}^k \arcs_\ell^]
 \qquad\mbox{and}\qquad
 \sum_{\ell=1}^n \arcs_\ell^[ = \sum_{\ell=1}^n \arcs_\ell^] = \arcs.
\end{equation}
Indeed, the arcs given by matching brackets are automatically disjoint.

Note that this representation as condensed bracket expression is particularly
useful for arc collections which contain large families of non-branched
nested arcs.

\subsection{...as a cleaved rainbow meander}

\begin{figure}
\centering
\begin{tikzpicture}[scale=0.99*\hsize/43cm]
\fill [color=background] (0.7,0) -- (10.3,0) arc (0:-180:4.8 and 2.4);
\fill [color=background] (10.7,0) -- (20.3,0) arc (0:180:4.8 and 2.4);
\draw [background, line width=5pt, ->] (5.5,0) arc (180:360:5 and 4);
\drawMeanderPath{1}{10,-7,8,-9,6,-1};
\drawMeanderPath{2}{5,-4,3,-2};
\draw [thin] (0,0) -- (11,0);
\begin{scope}[shift={(22,0)}]
  \drawMeanderPath{1}{10,-11,14,-7,8,-13,12,-9,6,-15,20,-1};
  \drawMeanderPath{2}{5,-16,17,-4,3,-18,19,-2};
  \draw [thin] (0,0) -- (21,0);
\end{scope}
\end{tikzpicture}
\caption{\label{figMeanderFlip}
Flip of the lower arc collection of a meander.}
\end{figure}
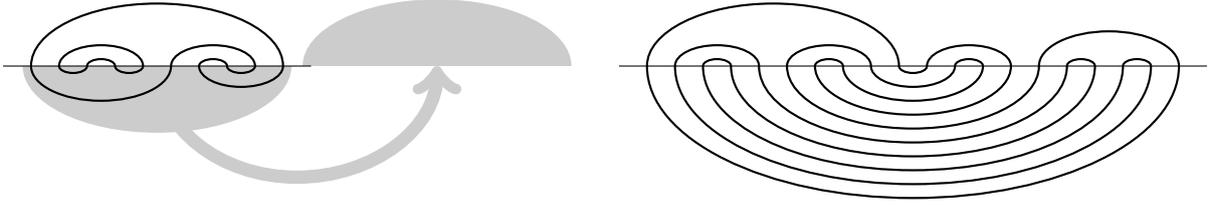

We are interested in the number of connected components,
i.e.~closed curves, of the meander. This number remains the same
if we ``simplify'' the lower arc collection of the meander by flipping
it to the upper part. More precisely, we rotate the lower arc collection
around a point on the horizontal axis to the right of the meander,
see figure \ref{figMeanderFlip}.
This operation doubles the number of intersection
points with the horizontal axis but replaces the lower arc collection by a
single non-branched family of nested arcs --- a \emph{rainbow family}.

In \cite{DiFrancescoGolinelliGuitter1997-Meander},
a meander is called a \emph{rainbow meander}
if the lower arcs form a single rainbow family, i.e.~are all nested.
A meander is called \emph{cleaved}
if none of the upper arcs connects a point $1 \le \ell \le \alpha$ on
the left half to a point  $\alpha < \tilde\ell \le 2\alpha$ on the right half
of the horizontal axis,
that is if the upper arc collection is split at the midpoint.
The flip, described above, then results in a cleaved rainbow meander.

Without loss of generality, from now on,
we assume that all meanders are rainbow meanders,
i.e.~have a single rainbow family as its lower arc collection.
The representations of permutations, discussed above,
take the respective forms:
\begin{equation}\label{eqFlippedRepresentations}
\begin{array}{rcl}
\pi^\arcsbelow &=& (1,2\arcs)(2,2\arcs-1)\cdots(\arcs,\arcs+1),\\
\pi^\arcsboth(k) &=& \left\{\begin{array}{ll}
   \pi^\arcsabove(k), & k \mbox{ odd},\\
   2\arcs-k+1, \quad & k \mbox{ even}.
\end{array}\right.
\end{array}
\end{equation}
As condensed bracket expression, the lower arc collection has the form
$((\arcs,\arcs))$ and can be omitted.
If necessary, we apply the flip.
The condensed bracket expression of the new upper arc collection is
the old one continued by the reflected old lower expression.
Specifically, for the example in
figures \ref{figMeanderHomotopy}, \ref{figMeanderFlip}, we obtain
\begin{equation}\label{eqBracketFlipExample}
((3,2),(2,3),(2,2),(2,1),(1,2)),
\end{equation}
see also the former representations (\ref{eqTranspositionExample},
\ref{eqBracketExample}, \ref{eqCondensedBracketExample}).

\begin{defn}[Meander]\label{defMeander}
We identify a meander with the condensed bracket expression
of its upper arc collection (after the flip, for non-rainbow meanders)
and use the notation
\begin{equation}\label{eqMeanderDef}
\meander \;=\; \meander((\arcs_1^[,\arcs_1^]),\ldots,(\arcs_n^[,\arcs_n^])),
\end{equation}
satisfying (\ref{eqCondensedBracketExpressionConditions}).
\end{defn}

Note that a given $n$-tuple (\ref{eqMeanderDef}) of pairs of positive integers
represents a flipped meander if, and only if, it is cleaved:
\begin{equation}\label{eqCleavedMeander}
\sum_{\ell=1}^k \arcs_\ell^[ \;=\; \sum_{\ell=1}^k \arcs_\ell^] \;=\;
\sum_{\ell=k+1}^n \arcs_\ell^[ \;=\; \sum_{\ell=k+1}^n \arcs_\ell^]
\;=\; \arcs/2,
\end{equation}
for an appropriate  $k$.
Otherwise, the inverse flip would create a meander curve with ``overhanging''
arcs from the upper to the lower side of the axis.
Such meanders can be interpreted as the intersection pattern of closed
Jordan curves with a half line instead of a line.
See again \cite{DiFrancescoGolinelliGuitter1997-Meander},
where this viewpoint is further developed.

\subsection{...as an element of a Temperley-Lieb algebra}

\begin{figure}
\centering
\begin{tikzpicture}[scale=0.8*\hsize/20cm, yscale=0.5]\small
\begin{scope}[ shift={(0,0)} ]
\node [left] at (-2,-4.5) {{\boldmath$e_0\!=\!1$:\strut}};
\draw [thin] (0,0.5) -- ++(0,-10) (3,0.5)-- ++(0,-10);
\draw [thick] \foreach \slY in {0,-1,-4,-5,-8,-9} { (0,\slY) -- ++(3,0) };
\node [left] at (0, 0) {$1$};
\node [left] at (0,-1) {$2$};
\node [left] at (0,-8) {$\alpha\!-\!1$};
\node [left] at (0,-9) {$\alpha$};
\node at (1.5,-2.5) {$\vdots$\strut};
\node at (1.5,-6.5) {$\vdots$\strut};
\end{scope}
\begin{scope}[ shift={(10,0)} ]
\node [left] at (-2,-4.5) {{\boldmath$e_\ell$:\strut}};
\draw [thin] (0,0.5) -- ++(0,-10) (3,0.5) -- ++(0,-10);
\draw [thick] \foreach \slY in {0,-1,-8,-9} { (0,\slY) -- ++(3,0) }
   (0,-5) -- ++( 1,0) arc(-90:90:0.2 and 0.5) -- ++(-1,0)
   (3,-5) -- ++(-1,0) arc(270:90:0.2 and 0.5) -- ++( 1,0);
\node [left] at (0, 0) {$1$};
\node [left] at (0,-1) {$2$};
\node [left] at (0,-4) {$\ell$};
\node [left] at (0,-5) {$\ell\!+\!1$};
\node [left] at (0,-8) {$\alpha\!-\!1$};
\node [left] at (0,-9) {$\alpha$};
\node at (1.5,-2.5) {$\vdots$\strut};
\node at (1.5,-6.5) {$\vdots$\strut};
\end{scope}
\end{tikzpicture}
\caption{\label{figTLgenerators}
Generators of the Temperley-Lieb algebra $TL_{\alpha}(q)$ as strand diagrams.}
\end{figure}
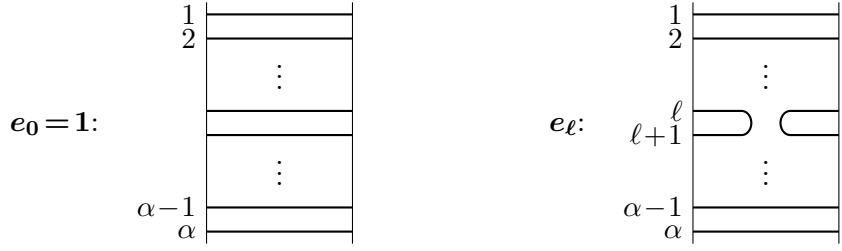

The multiplicative generators $e_0=1,e_1,\ldots,e_{\alpha-1}$
of a Temperley-Lieb algebra $TL_{\alpha}(q)$
\cite{TemperleyLieb1971-TemperleyLieb}
obey the relations
\begin{equation}\label{eqTLreleation}
\begin{array}{rcll}
e_\ell^2 &=& q e_\ell, &(a)\\
e_\ell e_k &=& e_k e_\ell, \qquad \mbox{if } |k-\ell| \ge 2, \quad &(b)\\
e_\ell e_{\ell\pm1} e_\ell &=& e_\ell. &(c)
\end{array}
\end{equation}
They can be visualized as strand diagrams, see figure \ref{figTLgenerators}.
Then, the strand diagram of a general product $e_{\ell_1}\cdots e_{\ell_n}$
is given as the concatenation of the individual strand diagrams of
$e_{\ell_1},\ldots,e_{\ell_n}$.
The properties (\ref{eqTLreleation}) allow isotopic transformations
of the strand diagrams.
Possible islands, i.e.~closed Jordan curves in the strand diagram,
can be removed and then appear as a pre-factor $q$ due to
(\ref{eqTLreleation}a).
Relations (\ref{eqTLreleation}) can be used to define a basis of reduced
elements written as pure products $e_{\ell_1}\cdots e_{\ell_n}$ without islands.

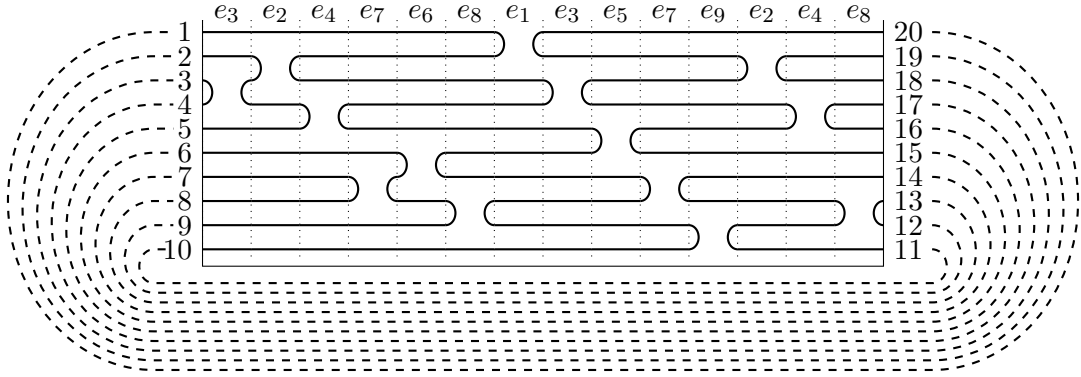
\begin{figure}
\centering
\begin{tikzpicture}[scale=0.8*\hsize/20cm, yscale=0.5]\small
\foreach \slX/\slE in
  {0/3,1/2,2/4,3/7,4/6,5/8,6/1,7/3,8/5,9/7,10/9,11/2,12/4,13/8} {
  \draw [thick] (\slX  ,-\slE) arc (-90:90:0.2 and 0.5);
  \draw [thick] (\slX+1,-\slE) arc (270:90:0.2 and 0.5);
  \foreach \slY in {2,3,...,\slE} {
    \ifnum\slY>\slE\else\draw [thick] (\slX,2-\slY) -- ++(1,0);\fi }
  \foreach \slY in {8,7,...,\slE} {
    \ifnum\slY<\slE\else\draw [thick] (\slX,-1-\slY) -- ++(1,0);\fi }
  \draw [thin, dotted] (\slX,0.5) -- ++(0,-10);
  \node [above] at (\slX.5,0) {$e_{\slE}$};
}
\draw [thin ] (0,0.5) -- ++(0,-10.2) -- ++(14,0) -- ++(0,10.2);
\foreach \slY in {0,...,9} {
  \draw [thick, dashed] (15,-\slY) arc (90:-90:3-\slY*0.3 and 7-\slY*0.7)
                      -- ++(-16,0) arc (270:90:3-\slY*0.3 and 7-\slY*0.7)
                      -- ++(0.4,0);
}
\foreach \slN in {1,...,10} {
  \node at (0,1-\slN) [left] {\slN}; }
\foreach \slN in {11,...,20} {
  \node at (14,\slN-20) [right] {\slN}; }
\end{tikzpicture}
\caption{\label{figTemperleyLieb}
Meander as the closure of an element of a Temperley-Lieb algebra.}
\end{figure}

A reduced element $e_{\ell_1}\cdots e_{\ell_n}$ becomes a
rainbow meander when we connect the left and right vertical boundaries
of the strand diagram by a rainbow family. This closure is illustrated in
figure \ref{figTemperleyLieb}, where we again obtain the meander
example (\ref{eqBracketFlipExample}) of figure \ref{figMeanderFlip}.
The horizontal line of the meander corresponds to the left and right
boundaries of the strand diagram of the Temperley-Lieb element,
glued at their bottom ends.

The trace $\mathrm{tr}(e)$ is defined as a linear function on
$TL_\alpha(q)$. It plays a crucial role in defining further operators on
the Temperley-Lieb algebra.
On products $e = e_{\ell_1}\cdots e_{\ell_n}$ the trace is
given by
\[
  \mathrm{tr}(e) \,:=\; q^{\components(e)},
\]
where $\components(e)$ is the number of connected components of
the strand diagram with identified endpoints of the same height
in the left and right boundary.
Without islands, this coincides with the number of Jordan curves in the
associated meander.
The ring element $q$ is the parameter of the Temperley-Lieb algebra.
See \cite{DiFrancescoGolinelliGuitter1997-Meander} for further background
on this correspondence.

\subsection{...as a Cartesian billiard\label{secBilliard}}

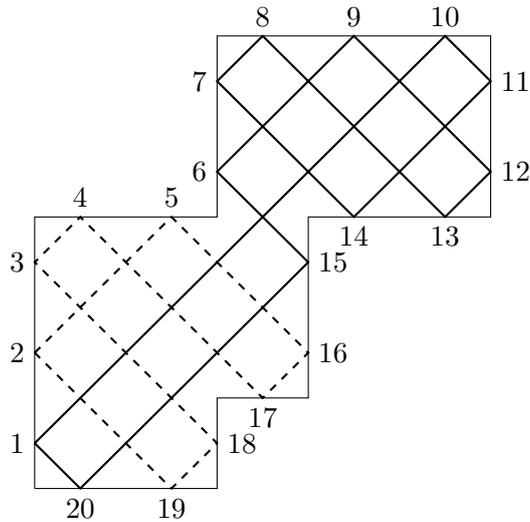
\begin{figure}
\centering
\begin{tikzpicture}[scale=0.45*\hsize/6cm]\small
\draw [thin]
  (0,0)--(0,3)--(2,3)--(2,5)--(5,5)--(5,3)--(3,3)--(3,1)--(2,1)--(2,0)--(0,0);
\draw [thick]
  (0,0.5)--(4.5,5)--(5,4.5)--(3.5,3)--(2,4.5)--(2.5,5)--(4.5,3)--(5,3.5)
         --(3.5,5)--(2,3.5)--(3,2.5)--(0.5,0)--(0,0.5);
\draw [thick, dashed]
  (0,1.5)--(1.5,3)--(3,1.5)--(2.5,1)--(0.5,3)--(0,2.5)--(2,0.5)--(1.5,0)
         --(0,1.5);
\node at (0,0.5) [left ]  {1};
\node at (0,1.5) [left ]  {2};
\node at (0,2.5) [left ]  {3};
\node at (0.5,3) [above]  {4};
\node at (1.5,3) [above]  {5};
\node at (2,3.5) [left ]  {6};
\node at (2,4.5) [left ]  {7};
\node at (2.5,5) [above]  {8};
\node at (3.5,5) [above]  {9};
\node at (4.5,5) [above] {10};
\node at (5,4.5) [right] {11};
\node at (5,3.5) [right] {12};
\node at (4.5,3) [below] {13};
\node at (3.5,3) [below] {14};
\node at (3,2.5) [right] {15};
\node at (3,1.5) [right] {16};
\node at (2.5,1) [below] {17};
\node at (2,0.5) [right] {18};
\node at (1.5,0) [below] {19};
\node at (0.5,0) [below] {20};
\end{tikzpicture}
\caption{\label{figBilliard}
Meander as trajectories of a Cartesian billiard.}
\end{figure}
\ignore{
\begin{figure}
\centering
\begin{tikzpicture}[scale=0.6*\hsize/20cm]\small
\draw [thin]  ( 0, 0)--( 6, 6)--(10, 2)--(14, 6)--(20, 0)
           -- (16,-4)--(12, 0)--( 8,-4)--( 6,-2)--( 4,-4)--( 0, 0);
\draw [thick] ( 1, 1)--(19, 1)--(19,-1)--(13,-1)--(13, 5)--(15, 5)
           -- (15,-3)--(17,-3)--(17, 3)--(11, 3)--(11,-1)--( 1,-1)--( 1, 1);
\draw [thick, dashed]  ( 3, 3)--( 9, 3)--( 9,-3)--( 7,-3)
                     --( 7, 5)--( 5, 5)--( 5,-3)--( 3,-3)--( 3, 3);
\node at ( 1, 1) [above]  {1\strut};
\node at ( 3, 3) [above]  {2\strut};
\node at ( 5, 5) [above]  {3\strut};
\node at ( 7, 5) [above]  {4\strut};
\node at ( 9, 3) [above]  {5\strut};
\node at (11, 3) [above]  {6\strut};
\node at (13, 5) [above]  {7\strut};
\node at (15, 5) [above]  {8\strut};
\node at (17, 3) [above]  {9\strut};
\node at (19, 1) [above] {10\strut};
\node at (19,-1) [below] {11\strut};
\node at (17,-3) [below] {12\strut};
\node at (15,-3) [below] {13\strut};
\node at (13,-1) [below] {14\strut};
\node at (11,-1) [below] {15\strut};
\node at ( 9,-3) [below] {16\strut};
\node at ( 7,-3) [below] {17\strut};
\node at ( 5,-3) [below] {18\strut};
\node at ( 3,-3) [below] {19\strut};
\node at ( 1,-1) [below] {20\strut};
\end{tikzpicture}
\caption{\label{figBilliard}
Meander as trajectories of a Cartesian billiard.}
\end{figure}
}

A Cartesian billiard is played on a compact region B in the plane.
The boundary of B consists of horizontal and vertical connections of
corner points on the integer lattice $\setZ\times\setZ$.
The billiard trajectories are piecewise linear flights
on the diagonal grid $\{(x,y)\,|\,x\pm y\in\setZ+\frac{1}{2}\}$
and hit the boundary polygon in half-integer midpoints
$\setZ\times(\setZ+\frac{1}{2})\;\cup\;(\setZ+\frac{1}{2})\times\setZ$
with the standard reflection rule.
See figure \ref{figBilliard} for an illustration.

In \cite{FiedlerCastaneda2012-Meander} the close relation of
Cartesian billiards and meanders has been studied.
If the boundary of the billiard region is a single curve without
self intersections (or, more generally, of self intersection only at
integer lattice points~---
removable by making the corners of the boundary polygon round)
then the billiard trajectories correspond to meander curves.
Indeed, we take any consecutive enumeration of the half integer midpoints
along the billiard boundary.
They represent the intersection points of the meander
with the horizontal line.
The two families of parallel pieces of the billiard trajectories represent,
respectively, the upper and lower arcs of the meander.
In particular, the closed trajectories of the Cartesian billiard are mapped
onto the closed Jordan curves of the meander.

Conversely, a cleaved rainbow meander
$\meander((\arcs_1^[,\arcs_1^]),\ldots,(\arcs_n^[,\arcs_n^]))$
can be easily represented by a Cartesian billiard.
Indeed, we construct the billiard boundary by starting at the origin
and attaching a horizontal or vertical unit interval for each of
the $2\alpha$ upper brackets of our meander representation:
On the first half, i.e.~for the first $\alpha$ brackets,
we go up for opening brackets and right for closing brackets.
Due to condition (\ref{eqCleavedMeander}),
we arrive at the point $(\alpha/2,\alpha/2)$,
and stay above the diagonal $x=y$.
On the second half, i.e.~for the last $\alpha$ brackets,
we go down for opening brackets and left for closing brackets.
We stay below the diagonal $x=y$ and arrive at the origin.
(The only possible self intersections are touching points on the diagonal.)
See figure \ref{figBilliard} for the representation of the meander
example (\ref{eqBracketFlipExample}) of figure \ref{figMeanderFlip}.

Without the cleavage (\ref{eqCleavedMeander}),
the construction has an additional twist.
For pairs of matching brackets on the same side of the midpoint,
we do the same as before.
For pairs of matching brackets on opposite sides of the midpoint,
we switch the rule for the bracket closer to the midpoint.
(We must exclude the case of brackets of the same distance to the midpoint,
which create a circle.)
If the opening bracket is closer to the midpoint,
we go right for the opening and left for the matching closing bracket,
switching the former rule for opening brackets of the first half.
If the closing bracket is closer to the midpoint,
we go up for the opening and down for the closing bracket,
switching the former rule for closing brackets of the second half.
This results in a closed billiard boundary without self intersections
(except, possibly, integer-lattice touching points which can be removed
by making the corners round),
provided the original meander is circle-free, i.e.~has no closed curve
consisting of only one upper and one lower arc.
See \cite{FiedlerCastaneda2012-Meander} for a complete proof.

\subsection{Bi-rainbow meanders}

\begin{figure}
\centering
\begin{tikzpicture}[scale=0.75*\hsize/29cm]
\draw [thin] (0,0) -- (29,0);
\drawMeanderArc{10}{1}
\drawMeanderArc{9}{2}
\drawMeanderArc{6}{5}
\node at (5.5,1) {$\vdots$\strut};
\node [right] at (5.5,1) {$\arcs_1$\strut};
\drawMeanderArc{28}{17}
\drawMeanderArc{27}{18}
\drawMeanderArc{23}{22}
\node at (22.5,1.25) {$\vdots$\strut};
\node [right] at (22.5,1.25) {$\arcs_n$\strut};
\node at (13.5,1) {\scalebox{2}{$\cdots$}};
\drawMeanderArc{1}{28}
\drawMeanderArc{2}{27}
\drawMeanderArc{13}{16}
\drawMeanderArc{14}{15}
\node at (14.5,-3.5) {\scalebox{2}{$\vdots$\strut}};
\node [right] at (14.5,-3.5) {$\displaystyle\;\arcs = \sum_1^n \arcs_\ell$};
\end{tikzpicture}
\caption{\label{figRainbowMeander}
General bi-rainbow meander.}
\end{figure}
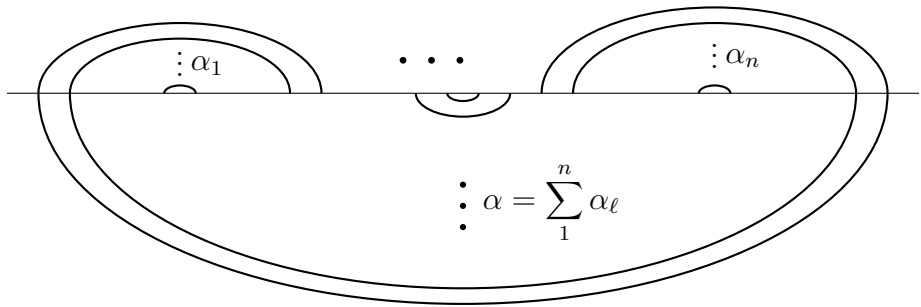

We have already called a single non-branched family of nested arcs
a \emph{rainbow family},
and a meander with a single lower rainbow family a rainbow meander.

If a meander consists only of rainbow families, that is if also the
upper arc collection consists only of non-branched families of nested arcs,
then we call the meander a bi-rainbow meander,
see figure \ref{figRainbowMeander}.

\begin{defn}[Bi-rainbow meander]\label{defRainbowMeander}
A bi-rainbow meander is a meander 
\begin{equation}\label{eqRainbowMeanderDef}
\rainbowMeander \;=\; \rainbowMeander(\arcs_1,\ldots,\arcs_n)
  \;:=\; \meander((\arcs_1,\arcs_1),\ldots,(\arcs_n,\arcs_n)),
\end{equation}
consisting of $\arcs=\sum_{\ell=1}^n \arcs_\ell$ upper arcs in $n$
rainbow families and one lower rainbow family of $\arcs$ nested arcs.
\end{defn}

Bi-rainbow meanders~---
or rather their collapsed variants introduced in section \ref{secCollapse}~---
represent the structure of seaweed algebras
\cite{DergachevKirillov2000-SeaweedAlgebras, CollMagnantWang2012-Meander}.
Here, the number of connected components is related to the index of the
associated seaweed algebra.

In \cite{FiedlerCastaneda2012-Meander, CollMagnantWang2012-Meander},
the question is raised, how to compute the number of connected components,
i.e.~closed curves,
\begin{equation}\label{eqRainbowComponentsDef}
\components \;=\; \components(\rainbowMeander)
\;=\; \components(\arcs_1,\ldots,\arcs_n),
\end{equation}
of a bi-rainbow meander.
In fact the easy expressions
\[
\components(\arcs_1,\arcs_2) = \gcd(\arcs_1,\arcs_2), \qquad
\components(\arcs_1,\arcs_2,\arcs_3) = \gcd(\arcs_1+\arcs_2,\arcs_2+\arcs_3),
\]
see (\ref{eqGcdOfRainbows}),
in terms of the greatest common divisors provoked the call
for a general ``closed'' formula.
This has also been the initial purpose of our investigation.


\section{Collapsed meanders}
\label{secCollapse}

In this section, we introduce the \emph{collapse} of a meander.
We start with a bi-rainbow meander
$\rainbowMeander \;=\; \rainbowMeander(\arcs_1,\ldots,\arcs_n)$,
drawn as arc collections in the plane, see figure \ref{figRainbowMeander}.
Above and below the horizontal axis, the meander splits the half plane into
connected components.
Coming from infinity, we colour each second component black:
If a path in the half plane from infinity into the component crosses an odd
number of arcs, then we colour this component.
The coloured components hit the horizontal axis in the intervals
$[2\ell-1, 2\ell]$, $\ell > 1$. In particular,
the coloured components above and below the horizontal axis match.
Furthermore, each arc bounds exactly one coloured component.

There are two types of coloured components. Most coloured components
are ``thickened arcs'' bounded by two (neighbouring) arcs of the same rainbow
family and two intervals on the axis.
The only exceptions are the innermost components of rainbow families with
an odd number of arcs: they are half disks bounded by an arc and an
interval on the axis.
See figure \ref{figCollapsedRainbow} for an illustration.

\begin{figure}
\centering
\includegraphics[width=\textwidth]{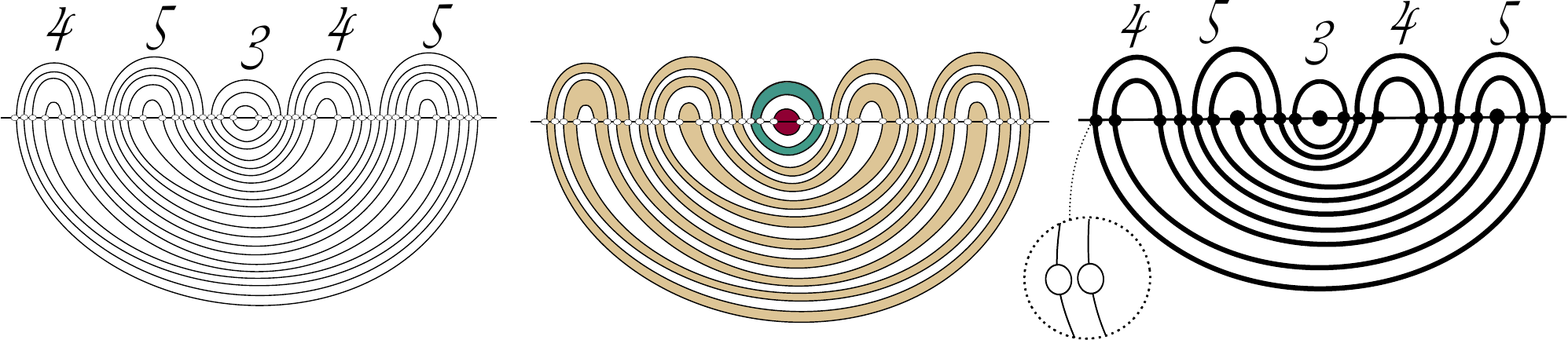}
\caption{\label{figCollapsedRainbow}
Collapse of the bi-rainbow meander $\rainbowMeander(4,5,3,4,5)$.
From left to right: $\rainbowMeander$, coloured domains, and
$\collapsedRainbowMeander$.
The collapsed bi-rainbow meander consists of one path, one cycle,
and an isolated point (counted as a second path).}
\end{figure}

\begin{defn}[Collapsed bi-rainbow meander]
The collapsed bi-rainbow meander,
denoted by $\collapsedRainbowMeander \;=\;
\collapsedRainbowMeander(\arcs_1,\ldots,\arcs_n)$,
arises when we collapse pairs of arcs of the
bi-rainbow meander $\rainbowMeander(\arcs_1,\ldots,\arcs_n)$ to single arcs,
that is when we collapse each coloured component, described above,
into an arc or a point.
The value $\arcs_\ell$ is the number of arcs in the $\ell$-th upper family of
$\rainbowMeander$ and the number of intersections with the axis
in the $\ell$-th upper family of $\collapsedRainbowMeander$.

The collapsed bi-rainbow meander is again composed of several rainbow arc
collections above and a single rainbow arc collection below the axis.
However, if $\arcs_\ell$ is odd, then the innermost ``arc'' of this upper
rainbow collection is a single point,
which we call \emph{semi-isolated}.
Similarly, if $\arcs=\sum_{\ell=1}^n \arcs_\ell$ is odd then the innermost
``arc'' of the lower rainbow collection is a semi-isolated point.
\end{defn}

Combining the arc collections of $\collapsedRainbowMeander$ above and below
the axis, we find again Jordan curves.
These curves can either be closed \emph{cycles} or open \emph{paths} ending
in semi-isolated points.
If $\arcs = 2\sum_{\ell=1}^{m-1} \arcs_\ell + \arcs_m$ is odd, then
the lower semi-isolated point coincides with the upper semi-isolated point
of the $m$-th rainbow family and becomes an isolated point of the
collapsed bi-rainbow meander. We consider such an isolated point to be a path.

\begin{thm}\label{thmCollapsedRainbowComponents}
The number $\components(\arcs_1,\ldots,\arcs_n)$ of connected components
(Jordan curves) of the bi-rainbow meander
$\rainbowMeander(\arcs_1,\ldots,\arcs_n)$
equals the sum of the number of paths and twice the number of cycles of the
collapsed bi-rainbow meander $\collapsedRainbowMeander(\arcs_1,\ldots,\arcs_n)$:
\begin{equation}\label{eqCollapsedRainbowComponents}
\components(\rainbowMeander) \;=\; \pathComponents(\collapsedRainbowMeander)
  + 2 \cycleComponents(\collapsedRainbowMeander).
\end{equation}
\end{thm}
\begin{proof}
We reverse the collapse from $\rainbowMeander$ to $\collapsedRainbowMeander$.
This replaces the curves of $\collapsedRainbowMeander$ by ``thick'' curves
which are non-intersecting domains in the plain.
The boundary curves of these domains are the Jordan curves of the original
bi-rainbow meander $\rainbowMeander$.
A thickened path is a simply-connected domain, its boundary a single
Jordan curve. A thickened cycle is a deformed ring domain,
its boundary consists of two Jordan curves.
\end{proof}

Note in particular the special case of an isolated point of
$\collapsedRainbowMeander$. Its ``thick'' counterpart is a disk bounded
by a single Jordan curve. Indeed, an isolated point is created
by the innermost arc of an upper family matching the innermost arc of
the lower family and thus forming a Jordan curve.

Let us count again the number of paths of the collapsed bi-rainbow meander.
Each path has two endpoints. These endpoints must be semi-isolated points
of the upper or lower arc collections.
Semi-isolated points are created by the innermost arcs of odd rainbow families.
We find:

\begin{cor}\label{thmRainbowComponentsParity}
The parity of the number $\components(\arcs_1,\ldots,\arcs_n)$ of connected
components (Jordan curves) of the bi-rainbow meander
$\rainbowMeander(\arcs_1,\ldots,\arcs_n)$
is given as half the number of odd rainbow families:
\begin{equation}\label{eqRainbowComponentsParity}
\components(\rainbowMeander) \;\equiv\; \pathComponents(\collapsedRainbowMeander)
\quad \pmod 2,
\end{equation}
where $2\pathComponents$ is the number of odd components of
$(\arcs_1,\ldots,\arcs_n,\arcs)$, $\arcs=\sum_{\ell=1}^n \arcs_\ell$.
\end{cor}

Note that $\arcs$ is odd if, and only if, the number of odd entries among
$(\arcs_1,\ldots,\arcs_n)$ is odd.
Thus, the number of odd components of $(\arcs_1,\ldots,\arcs_n,\arcs)$
is always even.

\begin{cor}\label{thmConnectedRainbowParity}
In particular,
a connected bi-rainbow meander $\rainbowMeander(\arcs_1,\ldots,\arcs_n)$
(given by a single Jordan curve)
must have exactly one or two odd entries among $(\arcs_1,\ldots,\arcs_n)$.
\end{cor}

\begin{figure}
\centering
\begin{tikzpicture}[scale=0.99*\hsize/43cm]
\draw [thick, fill=background]
  (1,0) \meanderPath{10,-11,14,-7,8,-13,12,-9,6,-15,20,-1}
  (2,0) \meanderPath{-19,18,-3,4,-17,16,-5,2};
\draw [thin] (0,0) -- (21,0);
\begin{scope}[shift={(22,0)}, thick/.style={line width=0.7mm}]
\draw [thick] (7.5,0) \meanderPath{-13.5,11.5,-9.5,1.5,-19.5,15.5,-5.5}
           -- (5.5,2.0);
\draw [thick] \meanderArc{3.5}{17.5} -- (17.5,1.0);
\draw [thin] (0,0) -- (21,0);
\end{scope}
\end{tikzpicture}
\caption{\label{figCollapsedMeander}
General collapsed meander with branched curves.}
\end{figure}
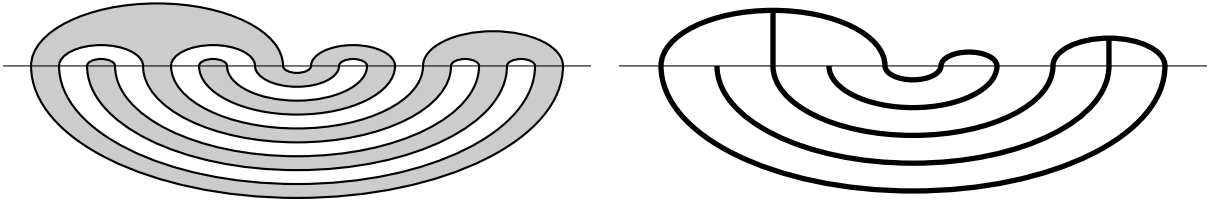

General meanders can also be collapsed in a similar fashion.
The resulting curves, however, will in general be branched.
See figure \ref{figCollapsedMeander} for the collapse of our
example (\ref{eqBracketFlipExample}).
The connected components of the collapsed meander must then be counted by the
number of components into which the plane is split by the branched curve.
We obtain a result similar to theorem \ref{thmCollapsedRainbowComponents}:

\begin{thm}\label{thmCollapsedGeneralComponents}
The number $\components((\arcs_1^[,\arcs_1^]),\ldots,(\arcs_n^[,\arcs_n^]))$
of connected components (Jordan curves) of the general meander
$\meander((\arcs_1^[,\arcs_1^]),\ldots,(\arcs_n^[,\arcs_n^]))$
equals the number of connected components of the
collapsed meander $\collapsedMeander$, counted by their multiplicity.
Here, the multiplicity of a (possibly branched) curve of $\collapsedMeander$
is given by the number of connected components of its complement
in the plane.
\end{thm}

A similar construction is used in \cite{CautosJacksonn2003-TemperleyLieb}
to relate meanders and their Temperley-Lieb counterparts to planar partitions.
In fact, its inverse is used to represent a planar partitions by a
Temperley-Lieb algebra.
Theorem \ref{thmCollapsedGeneralComponents} is found in the form
\[
  \components(\meander) \;=\;
  \components(\collapsedMeander)
  + \components(\setR^2\setminus\collapsedMeander) - 1.
\]
In other words, the number of Jordan curves of the meander equals the number
of coloured and bounded uncoloured regions.


\section{Nose retractions of bi-rainbow meanders\label{secNoseRetractions}}

Let $\rainbowMeander=(\arcs_1,\ldots,\arcs_n)$ be again an arbitrary bi-rainbow
meander with $n$ rainbow families of given numbers of arcs above and
one rainbow of $\arcs=\sum_{k=1}^n \arcs_k$ arcs below the horizontal line,
see figure \ref{figRainbowMeander}.
In this section, we discuss deformations of the meander $\rainbowMeander$ which
result again in a bi-rainbow meander with the same number
$\components(\rainbowMeander) = \components(\arcs_1,\ldots,\arcs_n)$
of connected components.
The general idea is to retract parts of upper rainbow families,
which we call \emph{noses}, through the horizontal axis.

Note, how the retraction of a single arc through the horizontal axis removes
two intersection points.
In the PDE application of section \ref{secMeanderShooting},
this corresponds to a saddle-node bifurcation in which the associated
stationary solutions of the PDE disappear.

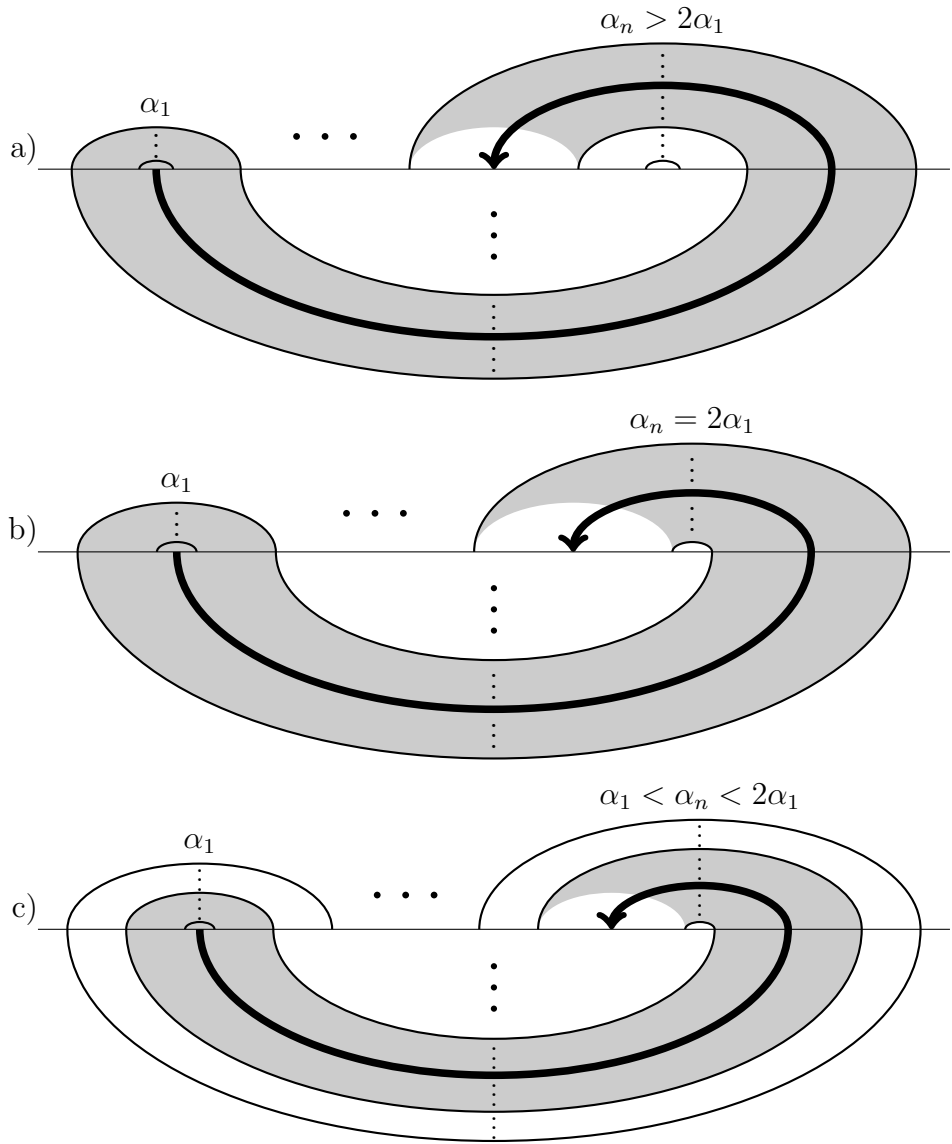
\begin{figure}
\centering
\begin{tikzpicture}[scale=0.75*\hsize/27cm]
\node at (0,0) {\makebox(0,0)[br]{a)\strut}};
\fill [background] (11,0) \meanderPath{16,21,-6,1,-26,11};
\draw [thin] (0,0) -- (27,0);
\drawMeanderPath{16}{21,-6,1,-26,11}
\drawMeanderArc{4}{3}
\drawMeanderArc{19}{18}
\draw [line width=1mm, ->] (3.5,0) \meanderPath{-23.5,13.5};
\node [above] at ( 3.5,1.25) {$\arcs_1$};
\node [above] at (18.5,3.75) {$\arcs_{n} > 2\arcs_1$};
\node at ( 3.5, 0.75 ) {$\vdots$\strut};
\node at (18.5, 0.75 ) {$\vdots$\strut};
\node at (18.5, 1.875) {$\vdots$\strut};
\node at (18.5, 3.125) {$\vdots$\strut};
\node at (13.5,-5.625) {$\vdots$\strut};
\node at (13.5,-4.375) {$\vdots$\strut};
\node at (13.5,-1.875) {\scalebox{2}{$\vdots$\strut}};
\node at ( 8.5, 1    ) {\scalebox{2}{$\cdots$\strut}};
\end{tikzpicture}
\begin{tikzpicture}[scale=0.75*\hsize/23cm]
\node at (0,0) {\makebox(0,0)[br]{b)\strut}};
\fill [background] (11,0) \meanderPath{16,17,-6,1,-22,11};
\draw [thin] (0,0) -- (23,0);
\drawMeanderPath{16}{17,-6,1,-22,11}
\drawMeanderArc{4}{3}
\draw [line width=1mm, ->] (3.5,0) \meanderPath{-19.5,13.5};
\node [above] at ( 3.5,1.25) {$\arcs_1$};
\node [above] at (16.5,2.75) {$\arcs_{n} = 2\arcs_1$};
\node at ( 3.5, 0.75 ) {$\vdots$\strut};
\node at (16.5, 0.875) {$\vdots$\strut};
\node at (16.5, 2.125) {$\vdots$\strut};
\node at (11.5,-4.625) {$\vdots$\strut};
\node at (11.5,-3.375) {$\vdots$\strut};
\node at (11.5,-1.375) {\scalebox{2}{$\vdots$\strut}};
\node at ( 8.5, 1    ) {\scalebox{2}{$\cdots$\strut}};
\end{tikzpicture}
\begin{tikzpicture}[scale=0.75*\hsize/31cm]
\node at (0,0) {\makebox(0,0)[br]{c)\strut}};
\fill [background] (17,0) \meanderPath{22,23,-8,3,-28,17};
\draw [thin] (0,0) -- (31,0);
\drawMeanderPath{22}{23,-8,3,-28,17}
\drawMeanderPath{15}{30,-1,10}
\drawMeanderArc{6}{5}
\draw [line width=1mm, ->] (5.5,0) \meanderPath{-25.5,19.5};
\node [above] at ( 5.5,2.25) {$\arcs_1$};
\node [above] at (22.5,3.75) {$\arcs_1 < \arcs_{n} < 2\arcs_1$};
\node at ( 5.5, 0.75 ) {\scalebox{0.9}{$\vdots$\strut}};
\node at ( 5.5, 1.75 ) {\scalebox{0.9}{$\vdots$\strut}};
\node at (22.5, 0.875) {\scalebox{0.9}{$\vdots$\strut}};
\node at (22.5, 2.125) {\scalebox{0.9}{$\vdots$\strut}};
\node at (22.5, 3.25 ) {\scalebox{0.9}{$\vdots$\strut}};
\node at (15.5,-5.625) {\scalebox{0.9}{$\vdots$\strut}};
\node at (15.5,-4.375) {\scalebox{0.9}{$\vdots$\strut}};
\node at (15.5,-6.75 ) {\scalebox{0.9}{$\vdots$\strut}};
\node at (15.5,-1.875) {\scalebox{2}{$\vdots$\strut}};
\node at (12.5, 1.2  ) {\scalebox{2}{$\cdots$\strut}};
\end{tikzpicture}
\caption{\label{figOuterRetraction}
Outer nose retractions of bi-rainbow meanders.}
\end{figure}

\begin{lem}[Outer nose retraction]\label{thmOuterNoseRetraction}
The number $\components(\arcs_1,\ldots,\arcs_n)$ of connected components of the
bi-rainbow meander $\rainbowMeander(\arcs_1,\ldots,\arcs_n)$ yields:
\begin{equation}\label{eqOuterNoseRetraction}
\begin{array}{l}
\components(\arcs_1,\arcs_2,\ldots,\arcs_{n-1},\arcs_n) \; = \\ \qquad =
\left\{\begin{array}{lll}
  \components(\arcs_2,\ldots,\arcs_{n-1}) + \arcs_1,
    & \arcs_1 = \arcs_n, & (a) \\
  \components(2\arcs_1-\arcs_n,\arcs_2,\ldots,\arcs_{n-1},\arcs_1),
    & \arcs_1 < \arcs_n < 2\arcs_1, & (b) \\
  \components(\arcs_2,\ldots,\arcs_{n-1},\arcs_1),
    & 2\arcs_1 = \arcs_n, & (c) \\
  \components(\arcs_2,\ldots,\arcs_{n-1},\arcs_1, \arcs_n-2\arcs_1), \qquad
    & 2\arcs_1 < \arcs_n. & (d)
\end{array}\right.
\end{array}\hspace*{-1em}
\end{equation}
By reflection, $\components(\arcs_1,\arcs_2,\ldots,\arcs_{n-1},\arcs_n)
  = \components(\arcs_n,\arcs_{n-1},\ldots,\arcs_2,\arcs_1)$,
the case $\arcs_n < \arcs_1$ is included.
\end{lem}

\begin{proof}
In case (a), $\arcs_1 = \arcs_n$,
we remove $\arcs_1$ outer cycles,
i.e.~connected components of the meander.

In case (d), $2\arcs_1 \le \arcs_n$,
we retract the full first (leftmost) upper rainbow family of size $\arcs_1$,
as shown in figure \ref{figOuterRetraction}a.
We hit the right half of the last (rightmost) upper rainbow family and retract
further until the retracted nose of size $\arcs_1$ arrives left of the
remaining $\arcs_n-2\arcs_1$ arcs of the last rainbow family.
In the boundary case (c), $2\arcs_1 = \arcs_n$, nothing remains of last
rainbow family, see figure \ref{figOuterRetraction}b.

In case (b), $\arcs_1 < \arcs_n < 2\arcs_1$,
we retract only the inner part of the first rainbow family,
such that we just hit the innermost arc of the last rainbow family,
as shown in figure \ref{figOuterRetraction}c.
Thereby, after retraction, the last rainbow family will remain
a (non-branched) rainbow family.
To hit the innermost arc, the retracted nose must consist
of $\arcs_n-\arcs_1$ arcs.
Therefore,
$\arcs_1-(\arcs_n-\arcs_1)=2\arcs_1-\arcs_n$ arcs of the first rainbow family
and $\arcs_n-(\arcs_n-\arcs_1)=\arcs_1$ arcs of the last rainbow family remain.
\end{proof}

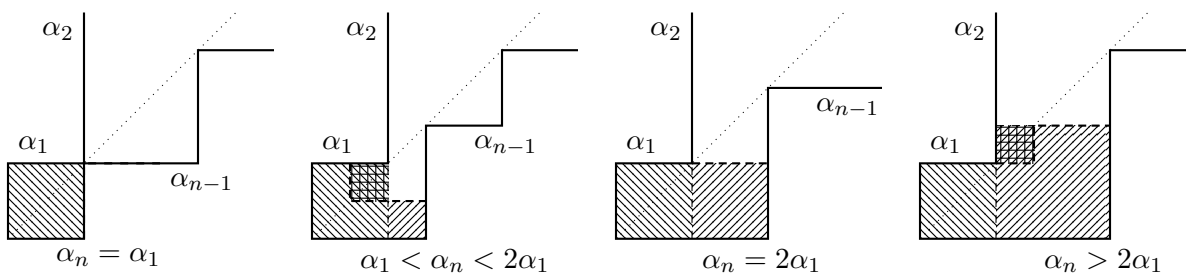
\begin{figure}
\centering
\begin{tikzpicture}[ scale=0.5 ]\small
\begin{scope}[ shift={(0,0)} ]
\draw [thick] (2,6)--(2,2)--(0,2)--(0,0)--(2,0)--(2,2)--(5,2)--(5,5)--(7,5);
\draw [thin, dotted] (0,0)--(6,6);
\draw [thick, dashed] (2,0)--(2,2)--(4,2);
\node [above right] at (0,2) {$\alpha_1$};
\node [below left ] at (2,6) {$\alpha_2$};
\node [below right] at (1,0) {$\alpha_n=\alpha_1$};
\node [below right] at (4,2) {$\alpha_{n-1}$};
\path [pattern=north west lines] (0,0) rectangle ++(2,2);
\end{scope}
\begin{scope}[ shift={(8,0)} ]
\draw [thick] (2,6)--(2,2)--(0,2)--(0,0)--(3,0)--(3,3)--(5,3)--(5,5)--(7,5);
\draw [thin, dotted] (0,0)--(6,6);
\draw [thin, dashed] (2,0)--(2,2);
\draw [thick, densely dashed] (1,2)--(1,1)--(3,1);
\node [above right] at (0,2) {$\alpha_1$};
\node [below left ] at (2,6) {$\alpha_2$};
\node [below right] at (4,3) {$\alpha_{n-1}$};
\node [below right] at (1,0) {$\alpha_1<\alpha_n<2\alpha_1$};
\path [pattern=north west lines] (0,0) rectangle ++(2,2);
\path [pattern=north east lines] (2,0) rectangle ++(1,1);
\path [pattern=grid] (1,1) rectangle ++(1,1);
\end{scope}
\begin{scope}[ shift={(16,0)} ]
\draw [thick] (2,6)--(2,2)--(0,2)--(0,0)--(4,0)--(4,4)--(7,4);
\draw [thin, dotted] (0,0)--(6,6);
\draw [thin, dashed] (2,0)--(2,2);
\draw [thick, densely dashed] (2,2)--(4,2);
\node [above right] at (0,2) {$\alpha_1$};
\node [below left ] at (2,6) {$\alpha_2$};
\node [below right] at (5,4) {$\alpha_{n-1}$};
\node [below right] at (2,0) {$\alpha_n=2\alpha_1$};
\path [pattern=north west lines] (0,0) rectangle ++(2,2);
\path [pattern=north east lines] (2,0) rectangle ++(2,2);
\end{scope}
\begin{scope}[ shift={(24,0)} ]
\draw [thick] (2,6)--(2,2)--(0,2)--(0,0)--(5,0)--(5,5)--(7,5);
\draw [thin, dotted] (0,0)--(6,6);
\draw [thin, dashed] (2,0)--(2,2) (2,3)--(3,3);
\draw [thick, densely dashed] (2,2)--(3,2)--(3,3)--(5,3);
\node [above right] at (0,2) {$\alpha_1$};
\node [below left ] at (2,6) {$\alpha_2$};
\node [below right] at (3,0) {$\alpha_n>2\alpha_1$};
\path [pattern=north west lines] (0,0) rectangle ++(2,2);
\path [pattern=north east lines] (2,0) rectangle ++(3,3);
\path [pattern=grid] (2,2) rectangle ++(1,1);
\end{scope}
\end{tikzpicture}
\caption{\label{figBilliardOuterRetraction}
Cutting of Cartesian billiards to resemble outer nose retractions.}
\end{figure}

In terms of Cartesian billiards, section \ref{secBilliard},
rainbow families which do not encompass the midpoint are represented
as triangles attached to the diagonal.
Cases (a) and (c) of (\ref{eqOuterNoseRetraction}) can be achieved
by cutting off
squares from the billiard domain,
see figure \ref{figBilliardOuterRetraction}.
The removal of squares which have three sides on the billiard boundary
does not change the connectivity of trajectories.
Indeed the exit point of an arbitrary trajectory on the square
coincides with its entry point, with reflected directions.
Cases (b) and (d) can be represented by the removal of two
and the subsequent attachment of one square.
In case (d), however, we need $\alpha_2>\alpha_n-2\alpha_1$ for the
second square to fit inside the billiard domain;
otherwise, the procedure fails.
The Cartesian billiard benefits at this point from the relation to
meander curves, where the nose retraction is always possible.

We already see that this lemma provides a strict reduction of the meander.
Therefore, its iteration will determine the number of connected components
after finitely many steps.
In section \ref{secAlgorithms}, we will improve the case $2\arcs_1 < \arcs_n$
to find an algorithm of logarithmic complexity.

In lemma \ref{thmSequenceConnected}, below, we will use one of the
inverse operations of (\ref{eqOuterNoseRetraction}),
\begin{equation}\label{eqInverseOuterNoseRetraction}
\components(\arcs_1,\ldots,\arcs_n) \;=\;
  \components(\arcs_n,\arcs_1,\ldots,\arcs_{n-1},2\arcs_n) \;=\;
  \components(2\arcs_n,\arcs_1,\ldots,\arcs_{n-1},2\arcs_n,\arcs_n),
\end{equation}
to construct a particular sequence of connected bi-rainbow meanders.

Instead of retracting the outer noses, as in the lemma above,
we now want to retract an inner nose.
The middle rainbow turns out to be a particular useful choice.
To determine the middle upper rainbow family, we define
\begin{equation}\label{eqMiddleRainbowLR}
\begin{array}{rcl}
L(\ell) &:=& \arcs_1+\arcs_2+\cdots+\arcs_{\ell-1},\\
R(\ell) &:=& \arcs_{\ell+1}+\cdots+\arcs_n,
\end{array}
\end{equation}
the total numbers of arcs left and right of the $\ell$-th upper rainbow family.
The index $m^*$ of the middle rainbow is now given as the unique value,
such that
\begin{equation}\label{eqMiddleRainbow}
L^* = L(m^*) < \arcs/2 \le L(m^*+1).
\end{equation}
For later reference, we note
\begin{equation}\label{eqMiddleRainbowR}
R^* \;=\; R(m^*) \le \arcs/2,
\qquad L^*+\alpha_{m^*}+R^* \;=\; \arcs.
\end{equation}

\begin{lem}[Inner nose retraction]\label{thmInnerNoseRetraction}
The number $\components(\arcs_1,\ldots,\arcs_n)$ of connected components of the
bi-rainbow meander $\rainbowMeander(\arcs_1,\ldots,\arcs_n)$, with $L^*, R^*, m^*$ as above,
yields:
\begin{equation}\label{eqInnerNoseRetraction}
\begin{array}{l}
\components(\arcs_1,\ldots,\arcs_n) \; = \\ \quad =
\left\{\begin{array}{lll}
\components(\arcs_1,\ldots,\arcs_{m^*-1},\arcs_{m^*+1},\ldots,\arcs_n)
  + \arcs_{m^*},
  & |L^*\!-\!R^*| = 0, & (a) \\
\components(\arcs_1,\ldots,\arcs_{m^*-1},\arcs_{m^*+1},\ldots,\arcs_n),
  & |L^*\!-\!R^*| = \arcs_{m^*}, & (b) \\
\components(\arcs_1,\ldots,\arcs_{m^*-1},\arcs_{m^*}-|L^*\!-\!R^*|,
  \arcs_{m^*+1},\ldots,\arcs_n),
  & \mbox{otherwise}. & (c)
\end{array}\right.
\end{array}\hspace*{-1em}
\end{equation}
\end{lem}

\begin{figure}
\centering
\begin{tikzpicture}[scale=0.75*\hsize/23cm]
\fill [background] (9,0) \meanderPath{-14,-15,20,-9};
\draw [thin] (6,0) -- (29,0);
\drawMeanderPath{14}{-15,20,-9}
\drawMeanderArc{24}{11}
\drawMeanderArc{18}{17}
\draw [line width=1mm, ->] (17.5,0) \meanderPath{-11.5};
\node [above] at (17.5,3.25) {$\arcs_{m^*}$};
\node at (17.5, 0.725) {$\vdots$\strut};
\node at (17.5, 2.25 ) {$\vdots$\strut};
\node at (14.5,-0.875) {$\vdots$\strut};
\node at (14.5,-2.125) {$\vdots$\strut};
\node at (14.5,-3.375) {\makebox(0,0){$\vdots$\strut}};
\node at ( 8  , 1.5  ) {\scalebox{2}{$\cdots$\strut}};
\node at (27  , 1.5  ) {\scalebox{2}{$\cdots$\strut}};
\end{tikzpicture}
\caption{\label{figInnerRetraction}
Inner nose retraction of a bi-rainbow meander.}
\end{figure}
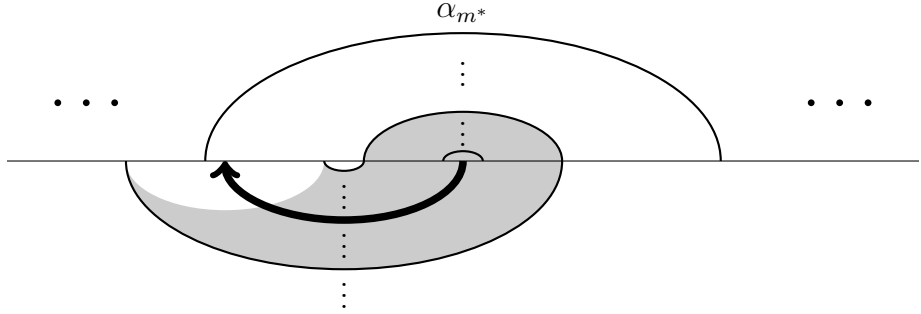

\begin{proof}
If $L^*=R^*$, then the middle family forms $m^*$ closed cycles, as each arc of
the family matches an arc of the lower family.

Otherwise, we retract the inner part of the $m^*$-th upper rainbow family,
such that we just hit the innermost arc of the lower rainbow family,
as in figure \ref{figInnerRetraction}. To achieve this, the
retracted nose must consist of $|L^*-R^*|$ arcs.
Then, $\arcs_{m^*}-|L^*-R^*|$ arcs remain in the middle upper rainbow family.
The lower family remains a (non-branched) rainbow family.

In the special case $|L^*-R^*| = \arcs_{m^*}$,
this procedure retracts the full $m^*$-th rainbow family.
In fact, in this case, $R^*=\arcs/2$ and the midpoint lies between the
$m^*$-th and its right neighbouring family.
Either of both families could be removed.
\end{proof}

\begin{figure}
\centering
\begin{tikzpicture}[ scale=0.5 ]\small
\begin{scope}[ shift={(-6.5,1)} ]
\draw [thick] (-5,0)--(1.5,0)--(1.5,1.5)--(0,1.5)--(0,-5)--(-5,-5);
\draw [thin, dotted] (-5,-5)--(1.5,1.5);
\node [above right] at (0,1.5) {$\alpha_{m^*}$};
\node [above right] at (-5, 0) {$\alpha_{m^*-1}$};
\node [above left ] at ( 0,-5) {$\alpha_{m^*+1}$};
\node [below left ] at ( 0, 0) {$R-L=0$};
\path [pattern=north west lines] (0,0) rectangle ++(1.5,1.5);
\end{scope}
\begin{scope}[ shift={(0,0)} ]
\draw [thick] (-4,0)--(0,0)--(0,3)--(2,3)--(2,-3)--(-3,-3)--(-3,-4);
\draw [thin, dotted] (-4,-4)--(2,2);
\draw [thick, densely dashed] (0,1)--(2,1);
\node [above right] at ( 0, 3) {$\alpha_{m^*}$};
\node [above right] at (-4, 0) {$\alpha_{m^*-1}$};
\node [above left ] at ( 2,-3) {$\alpha_{m^*+1}>R-L$};
\node [below left ] at ( 2, 0) {$0<R-L<\alpha_{m^*}$};
\path [pattern=north west lines] (0,1) rectangle ++(2,2);
\end{scope}
\begin{scope}[ shift={(7,0)} ]
\draw [thick] (-4,0)--(0,0)--(0,3)--(3,3)--(3,-2)--(-2,-2)--(-2,-4);
\draw [thin, dotted] (-4,-4)--(3,3);
\draw [thick, densely dashed] (0,0)--(3,0);
\node [above right] at ( 0, 3) {$\alpha_{m^*}$};
\node [above right] at (-4, 0) {$\alpha_{m^*-1}$};
\node [below right] at (-2,-2) {$\alpha_{m^*+1}>R-L$};
\node [below left ] at ( 3, 0) {$R-L=\alpha_{m^*}$};
\path [pattern=north west lines] (0,0) rectangle ++(3,3);
\end{scope}
\begin{scope}[ shift={(14,-1)} ]
\draw [thick] (-3,0)--(0,0)--(0,4)--(3,4)--(3,2)--(2,2)--(2,1)--(1,1)--(1,-3)
            --(-3,-3);
\draw [thin, dotted] (-3,-3)--(3,3);
\draw [thin, dashed] (0,1)--(3,1)--(3,2);
\node [above right] at ( 0, 4) {$\alpha_{m^*}$};
\node [above right] at (-3, 0) {$\alpha_{m^*-1}$};
\node [above left ] at ( 1,-3) {$\alpha_{m^*+3}$};
\path [pattern=dots] (0,1) rectangle ++(3,3);
\draw [ shift={(2.5,1.5)}, scale=0.5ex/1cm, x={(1,-0.3)}, y={(0.3,1)}, very thick ] (0,0)
         -- ++(0,-15) -- ++(5,8) -- ++(0,-12) ++(-1,1) -- ++(-1,1)
         -- ++(2,-7) -- ++(2,7) -- ++(-2,-2) -- ++(-1,1);
\end{scope}
\end{tikzpicture}
\caption{\label{figBilliardInnerRetraction}
Cutting of Cartesian billiards to resemble inner nose retractions.}
\end{figure}
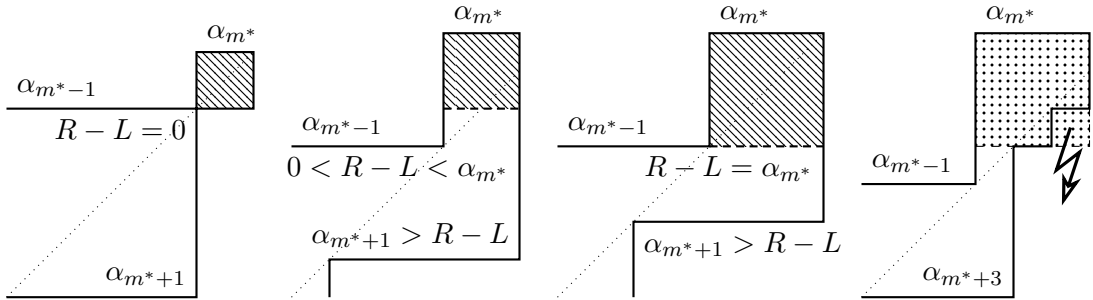

We can again try to rephrase the inner nose retraction in terms of
Cartesian billiards, section \ref{secBilliard}.
Note, that the middle rainbow $m^*$ contains the upper arcs which encompass
the midpoint.
It is the only rainbow family which is not represented by a triangle
over the diagonal.
We find simple cuts of single squares,
see figure \ref{figBilliardInnerRetraction}.
Cases (b) and (c) require, however, that the new middle family
is the old one or a direct neighbour.
Otherwise, i.e.~if the neighbouring family is too small,
there is no full square available,
as seen in the last picture of figure \ref{figBilliardInnerRetraction}.
Again, the meander view point is the preferred one.

Note also the inverse operation of (\ref{eqInnerNoseRetraction})(c).
For arbitrary $\ell$ we find
\begin{equation}\label{eqInverseInnerNoseRetraction}
\components(\arcs_1,\ldots,\arcs_n) \;=\;
  \components(\arcs_1,\ldots,\arcs_{\ell-1},\arcs_\ell+|L(\ell)-R(\ell)|,
    \arcs_{\ell+1},\ldots,\arcs_n).
\end{equation}
Indeed, the $\ell$-th family of size $\arcs_\ell+|L(\ell)-R(\ell)|$ becomes
the middle rainbow family $m^*$ on the right-hand side,
without changing $L$ and $R$.


\section{Non-/existing gcd formulae\label{secNoGcd}}

In this section, we try to express the number
$\components(\rainbowMeander) = \components(\arcs_1,\ldots,\arcs_n)$
of connected components of a bi-rainbow meander
$\rainbowMeander=(\arcs_1,\ldots,\arcs_n)$
by the greatest common divisor of expressions in $\arcs_\ell$.
Indeed, for $n\le3$:
\begin{equation}\label{eqGcdOfRainbowsAdv}
\begin{array}{lcl}
\components(\arcs_1) &=& \arcs_1, \\
\components(\arcs_1,\arcs_2) &=& \gcd(\arcs_1, \arcs_2), \\
\components(\arcs_1,\arcs_2,\arcs_3) &=& \gcd(\arcs_1+\arcs_2,\arcs_2+\arcs_3),
\end{array}
\end{equation}
see theorem \ref{thmGcdOfRainbows} below.

Before we show that there do not exist similar expressions of $\components$
for $n\ge 4$ in theorem \ref{thmNoGcd},
we establish a particular family of examples and a scaling property of
$\components$ in the following two preparatory lemmata.

\begin{lem}[Sequence of connected meanders]\label{thmSequenceConnected}
Let $\tilde{\arcs} \ge 2$, $\arcs^* \ge 1$, and $\tilde{n}\ge 0$
be arbitrary integers.
Then the bi-rainbow meanders
\begin{equation}\label{eqSequenceConnected}
\begin{array}{lll}
\rainbowMeander( \underbrace{2\tilde{\arcs},\ldots,2\tilde{\arcs}}_{\tilde{n}},
   \tilde{\arcs}-1, \arcs^*,
   \underbrace{2\tilde{\arcs},\ldots,2\tilde{\arcs}}_{\tilde{n}},
   \tilde{\arcs} )
   & , & n = 2\tilde{n}+3 \mbox{ odd}, \\
\rainbowMeander( \tilde{\arcs},
   \underbrace{2\tilde{\arcs},\ldots,2\tilde{\arcs}}_{\tilde{n}},
   \tilde{\arcs}-1, \arcs^*,
   \underbrace{2\tilde{\arcs},\ldots,2\tilde{\arcs}}_{\tilde{n}+1})
   & , & n = 2\tilde{n}+4 \mbox{ even},
\end{array}
\end{equation}
are connected, i.e.~$\components(\rainbowMeander) = 1$.
In particular, the bi-rainbow meander
$\rainbowMeander( \tilde{\arcs}, \tilde{\arcs}-1, \arcs^*, 2\tilde{\arcs} )$
with four upper families is connected.
\end{lem}
\begin{proof}
We proof the connectedness of the meanders by induction over $n$.
The bi-rainbow meander $\rainbowMeander(\tilde{\arcs}-1, \arcs^*, \tilde{\arcs})$
is connected due to the $\gcd$-formula (\ref{eqGcdOfRainbowsAdv}).
The inverse outer nose retraction (\ref{eqInverseOuterNoseRetraction})
yields
\[
\begin{array}{rcl}
\components( \underbrace{2\tilde{\arcs},\ldots,2\tilde{\arcs}}_{\tilde{n}},
   \tilde{\arcs}-1, \arcs^*,
   \underbrace{2\tilde{\arcs},\ldots,2\tilde{\arcs}}_{\tilde{n}},
   \tilde{\arcs} )
&=&
\components( \tilde{\arcs},
   \underbrace{2\tilde{\arcs},\ldots,2\tilde{\arcs}}_{\tilde{n}},
   \tilde{\arcs}-1, \arcs^*,
   \underbrace{2\tilde{\arcs},\ldots,2\tilde{\arcs}}_{\tilde{n}+1})
\\ &=&
\components( \underbrace{2\tilde{\arcs},\ldots,2\tilde{\arcs}}_{\tilde{n}+1},
   \tilde{\arcs}-1, \arcs^*,
   \underbrace{2\tilde{\arcs},\ldots,2\tilde{\arcs}}_{\tilde{n}+1},
   \tilde{\arcs} ).
\end{array}
\]
This proves the claim with the above base clause $\tilde{n}=0$.
\end{proof}

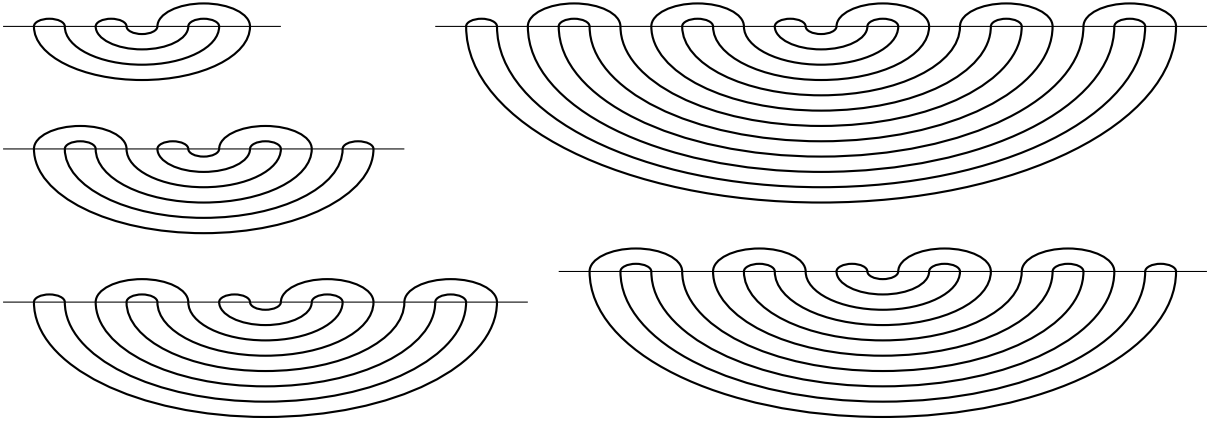
\begin{figure}
\centering
\setlength{\unitlength}{0.025\textwidth}
\begin{tikzpicture}[scale=0.99*\hsize/39cm]
\begin{scope}[shift={(4,0)}]
\draw [thin] (-4,0) -- (5,0);
\drawMeanderArc    { -2}{ -3}
\drawMeanderArc    {  0}{ -1}
\drawMeanderRainbow{  4}{  1}{ 4}
\drawMeanderRainbow{ -3}{  4}{ 8}
\end{scope}
\begin{scope}[shift={(6,-4)}]
\draw [thin] (-6,0) -- (7,0);
\drawMeanderRainbow{ -2}{ -5}{ 4}
\drawMeanderArc    {  0}{ -1}
\drawMeanderRainbow{  4}{  1}{ 4}
\drawMeanderArc    {  6}{  5}
\drawMeanderRainbow{ -5}{  6}{12}
\end{scope}
\begin{scope}[shift={(8,-9)}]
\draw [thin] (-8,0) -- (9,0);
\drawMeanderArc    { -6}{ -7}
\drawMeanderRainbow{ -2}{ -5}{ 4}
\drawMeanderArc    {  0}{ -1}
\drawMeanderRainbow{  4}{  1}{ 4}
\drawMeanderRainbow{  8}{  5}{ 4}
\drawMeanderRainbow{ -7}{  8}{16}
\end{scope}
\begin{scope}[shift={(28,-8)}]
\draw [thin] (-10,0) -- (11,0);
\drawMeanderRainbow{ -6}{ -9}{ 4}
\drawMeanderRainbow{ -2}{ -5}{ 4}
\drawMeanderArc    {  0}{ -1}
\drawMeanderRainbow{  4}{  1}{ 4}
\drawMeanderRainbow{  8}{  5}{ 4}
\drawMeanderArc    { 10}{  9}
\drawMeanderRainbow{ -9}{ 10}{20}
\end{scope}
\begin{scope}[shift={(26,0)}]
\draw [thin] (-12,0) -- (13,0);
\drawMeanderArc    {-10}{-11}
\drawMeanderRainbow{ -6}{ -9}{ 4}
\drawMeanderRainbow{ -2}{ -5}{ 4}
\drawMeanderArc    {  0}{ -1}
\drawMeanderRainbow{  4}{  1}{ 4}
\drawMeanderRainbow{  8}{  5}{ 4}
\drawMeanderRainbow{ 12}{  9}{ 4}
\drawMeanderRainbow{-11}{ 12}{24}
\end{scope}
\end{tikzpicture}
\caption{\label{figSequenceConnected}
Sequence of connected meanders generated from $\rainbowMeander(1,1,2)$
by iterated inverse outer nose retractions.}
\end{figure}

Note, how $\arcs^*$ always represents the middle upper family,
as introduced in (\ref{eqMiddleRainbow}), with $L^* = R^*-1$.
See figure \ref{figSequenceConnected} for an illustration.

\begin{lem}[Scaling]\label{thmComponentsScaling}
Let $\rainbowMeander(\arcs_1, \ldots, \arcs_n)$ be an arbitrary bi-rainbow meander
and $\lambda>0$ a positive integer.
Then, the number of connected components of the bi-rainbow meander
$\lambda\rainbowMeander:= \rainbowMeander(\lambda\arcs_1,\ldots,\lambda\arcs_n)$
scales linearly:
\begin{equation}\label{eqComponentsScaling}
\components(\lambda\rainbowMeander) \;=\;
\components(\lambda\arcs_1, \ldots, \lambda\arcs_n) \;=\;
\lambda \components(\arcs_1, \ldots, \arcs_n) \;=\;
\lambda \components(\rainbowMeander).
\end{equation}
\end{lem}
\begin{proof}
Let $(1,\ldots,2\arcs)$ again denote the intersections of the original bi-rainbow
meander $\rainbowMeander$ with the horizontal axis.
The scaled meander $\lambda\rainbowMeander$ replaces each arc of
$\rainbowMeander$ by $\lambda$ arcs.
If the original arc of $\rainbowMeander$ connects $a$ and $b$ on the axis,
then the corresponding arcs of $\lambda\rainbowMeander$ connect
$\lambda (a-1) + \ell$ with $\lambda b + 1 - \ell$,
for $\ell = 1,\ldots,\lambda$.
Furthermore, each arc of $\rainbowMeander$ connects an odd with an even point,
see (\ref{eqTranspositionOddEven}).
Therefore, $\lambda\rainbowMeander$ consists of $\lambda$ copies of
$\rainbowMeander$.
Indeed, each copy intersects the horizontal axis in one of the sets
$\{\ell,2\lambda+1-\ell,2\lambda+\ell,4\lambda+1-\ell,4\lambda+\ell,
  \ldots,2\arcs+1-\ell\}$,
$\ell=1,\ldots,\lambda$.
This immediately yields the scaling (\ref{eqComponentsScaling}) of the
number of connected components.
\end{proof}

We are now well prepared to prove the theorem claimed in the introduction:

\begin{thm}\label{thmNoGcd}
Let $n\ge4$ be given.
Then there do not exist homogeneous polynomials
$f_1,f_2 \in \setZ[x_1,\ldots,x_n]$ of arbitrary degree
with integer coefficients such that the
number of connected components $\components(\arcs_1,\ldots,\arcs_n)$
of every bi-rainbow meander $\rainbowMeander(\arcs_1,\ldots,\arcs_n)$
is given by the $\gcd(f_1(\arcs_1,\ldots,\arcs_n),f_2(\arcs_1,\ldots,\arcs_n))$.
In other words: to every choice of polynomials $f_1, f_2$,
we find a counterexample.
\end{thm}
\begin{proof}
We assume the contrary: let $n\ge 4$ and $f_1, f_2 \in \setZ[x_1,\ldots,x_n]$
be homogeneous polynomials with integer coefficients, such that for all
bi-rainbow meanders $\rainbowMeander = \rainbowMeander(\arcs_1,\ldots,\arcs_n)$
\begin{equation}\label{eqGcdAssumption}
\components(\arcs_1,\ldots,\arcs_n) \;=\;
\gcd(f_1(\arcs_1,\ldots,\arcs_n),f_2(\arcs_1,\ldots,\arcs_n)).
\end{equation}
We shall find the contradiction in four steps:
\begin{enumerate}[itemsep=0ex,topsep=0ex,parsep=0ex]
\item Show that one of $f_1, f_2$ must have degree one, i.e.~is linear.
\item Show that both of $f_1, f_2$ must have degree one, i.e.~are linear.
\item Find conditions on the parity of the coefficients of $f_1, f_2$.
\item Show the contradiction by the pigeonhole principle.
\end{enumerate}

\emph{Step 1.
[Show that one of $f_1, f_2$ must have degree one, i.e.~is linear.]}
Let $d_j \in \setN$ denote the degree of $f_j$, $j=1,2$. Then,
for all positive integers $\lambda$,
\[
f_j(\lambda\arcs_1, \ldots, \lambda\arcs_n)
  \;=\; \lambda^{d_j}(\arcs_1, \ldots, \arcs_n).
\]
Take an arbitrary bi-rainbow meander
$\rainbowMeander(\bar{\arcs}) := \rainbowMeander(\arcs_1,\ldots,\arcs_n)$
and $\lambda$ co-prime to $f_1(\bar\arcs)$ and $f_2(\bar\arcs)$.
The scaling lemma \ref{thmComponentsScaling} and
assumption (\ref{eqGcdAssumption}) then yield
\[
\begin{array}{rcl}
\lambda \components(\bar\arcs)
 &=&  \components(\lambda\bar\arcs)
\;=\; \gcd(f_1(\lambda\bar\arcs),f_2(\lambda\bar\arcs))
\;=\; \lambda^{\min(d_1,d_2)}\gcd(f_1(\bar\arcs),f_2(\bar\arcs))
\\&=& \lambda^{\min(d_1,d_2)}\components(\bar\arcs).
\end{array}
\]
Therefore, $\min(d_1,d_2) = 1$ and one of the polynomials must indeed be linear.

\emph{Step 2.
[Show that both of $f_1, f_2$ must have degree one, i.e.~are linear.]}
Without loss of generality, $d_1=1$ and $d_2 \ge 1$, by step 1.
Let $\rainbowMeander(\bar{\arcs})$ be the connected bi-rainbow meander of lemma
\ref{thmSequenceConnected} with $n$ upper rainbow families.
Indeed, the sequence (\ref{eqSequenceConnected}) contains one element
for each $n\ge 4$.
Again, we use the scaling lemma \ref{thmComponentsScaling} and
assumption (\ref{eqGcdAssumption}) to obtain
\[
\begin{array}{rcl}
\lambda
 &=&  \lambda \components(\bar\arcs)
\;=\; \components(\lambda\bar\arcs)
\;=\; \gcd(f_1(\lambda\bar\arcs),f_2(\lambda\bar\arcs))
\;=\; \gcd(\lambda f_1(\bar\arcs), \lambda^{d_2} f_2(\bar\arcs))
\\&=& \lambda \gcd(f_1(\bar\arcs), \lambda^{d_2-1} f_2(\bar\arcs)).
\end{array}
\]
Observe that $f_1$ must depend on $\arcs^*$,
the size of the middle arc collection of $\bar\arcs$.
Indeed,
taking a different bi-rainbow meander which only keeps the family $\arcs^*$,
\[
\begin{array}{rcll}
\components(1,\ldots,1,\arcs^*,1,\ldots,1) &=& \arcs^*+\frac{n-1}{2}, &
\mbox{for odd } n, \mbox{ and} \\
\components(1,\ldots,1,\arcs^*,1,\ldots,1,2) &=& \arcs^*+\frac{n-2}{2}, &
\mbox{for even } n,
\end{array}
\]
 become arbitrarily large for $\arcs^*\to\infty$.
Here, we have chosen $L^*=R^*$,
compare with lemma \ref{thmInnerNoseRetraction}.
As $f_1$ is linear, this forces the coefficient of $\arcs^*$ of $f_1$
to be non-zero.

Now, we choose $\alpha^*$ large enough, such that $f_1(\bar\arcs) > 1$.
Then we select $\lambda = f_1(\bar\arcs)$ to find
\[
\begin{array}{rcl}
\lambda
 &=& \lambda \gcd(f_1(\bar\arcs), \lambda^{d_2-1} f_2(\bar\arcs))
 \;=\; \lambda^2, \qquad \mbox{for\ } d_2 \ge 2.
\end{array}
\]
This is a contradiction. Therefore, $d_2=1$ as claimed.

\emph{Step 3.
[Conditions on the parity of the coefficients of $f_j$.]}
Let
\[
  f_j(\arcs_1,\ldots,\arcs_n) \;=\; \sum_{\ell=1}^n f_{j,\ell} \arcs_\ell,
  \qquad
  f_{j,\ell} \in \setZ,
\]
be the homogeneous polynomials of degree one.
From corollary \ref{thmRainbowComponentsParity} we know
that $\components(\bar\arcs)$ is odd for arbitrary
$\bar\arcs = (\arcs_1,\ldots,\arcs_n)$ with exactly one or two odd components.
The parity (mod 2) of assumption (\ref{eqGcdAssumption}) applied to
bi-rainbow meanders with exactly one odd component $\arcs_\ell$ yields
\[
  1 \;\equiv\; \gcd( f_{1,\ell}, f_{2,\ell} ) \pmod 2.
\]
Applied to bi-rainbow meanders with exactly two odd components
$\arcs_k,\arcs_\ell$, it yields
\[
  1 \;\equiv\; \gcd( f_{1,k}+f_{1,\ell}, f_{2,k}+f_{2,\ell} ) \pmod 2,
  \qquad k\neq\ell.
\]
Thus, for arbitrary $k\neq\ell$, the following conditions must hold:
\begin{equation}\label{eqParity}
\begin{array}{rcll}
(f_{1,\ell}, f_{2,\ell}) &\not\equiv& (0,0) &\pmod 2, \\
(f_{1,\ell}, f_{2,\ell}) &\not\equiv& (f_{1,k}, f_{2,k}) &\pmod 2.
\end{array}
\end{equation}

\emph{Step 4.
[Contradiction by the pigeonhole principle.]}
The first condition of (\ref{eqParity}) leaves only three possibilities
for $(f_{1,\ell}, f_{2,\ell})$:
\[
  \{\; (0,1), \; (1,0), \; (1,1) \;\} \pmod 2.
\]
If $n \ge 4$ then one of the three choices must appear more than once, say at
$k$ and $\ell$. But this violates the second condition of (\ref{eqParity}).
This is the final contradiction to the initial assumption and proves
the impossibility of a $\gcd$-formula (\ref{eqGcdAssumption}).
\end{proof}


\section{Euclidean-like algorithms\label{secAlgorithms}}

Nose retractions, as introduced in section \ref{secNoseRetractions},
have been used
before to establish finite algorithms on meander curves
\cite{CollMagnantWang2012-Meander}.
Here, however, we will improve the nose retractions
(\ref{eqOuterNoseRetraction}) and (\ref{eqInnerNoseRetraction}) to
establish rigorous bounds on the complexity of the resulting algorithms.
This will show a striking similarity to the calculation of the greatest common
divisor by the Euclidean algorithm.

\begin{prop}\label{thmGcdOfRainbows}
The number of connected components of bi-rainbow meanders with
less than four upper families is given by
\begin{equation}\label{eqGcdOfRainbows}
\begin{array}{lcl}
\components(\arcs_1) &=& \arcs_1, \\
\components(\arcs_1,\arcs_2) &=& \gcd(\arcs_1, \arcs_2), \\
\components(\arcs_1,\arcs_2,\arcs_3) &=& \gcd(\arcs_1+\arcs_2,\arcs_2+\arcs_3),
\end{array}
\end{equation}
where $\gcd$ denotes the greatest common divisor.
\end{prop}

\begin{proof}
The proof is easily done by induction over $\arcs=\sum \arcs_k$ using either
nose retraction (\ref{eqOuterNoseRetraction}) or (\ref{eqInnerNoseRetraction}).
\end{proof}

Note that the greatest common divisor is an abbreviation for the
Euclidean algorithm:
\begin{equation}\label{eqEuclideanAlgorithm}
\gcd(a_1,a_2) \;=\; \gcd(a_2,a_1) \;=\; \left\{
\begin{array}{lll}
a_1 &,& a_1=a_2,\\
\gcd(a_1, \remainder(a_2,a_1)) &,& a_1<a_2.
\end{array}\right.
\end{equation}
Here, $\remainder(a_2,a_1)$ denotes the remainder of the integer
division $a_2/a_1$.
This algorithm stops after $\Ord(\log a_1 + \log a_2)$ steps.
Indeed, $\remainder(a_2,a_1) < a_2/2$.
The number of bits needed to encode the problem is strictly decreased
in each step.
Here, we assume the elementary operations of (\ref{eqEuclideanAlgorithm})
to be of complexity $\Ord(1)$.
Complexity of arithmetic of large integers could be considered
but is not our focus here.

Turning back to the nose-retraction algorithm, denote
\[
\bits \;=\; \sum_{\ell=1}^n \Big\lceil \log_2 (\arcs_\ell+1) \Big\rceil \;=\;
\Ord\left( \sum_{\ell=1}^n \log \arcs_\ell \right)
\]
the number of bits needed to encode the bi-rainbow meander
$\rainbowMeander(\arcs_1,\ldots,\arcs_n)$.
We want to improve the nose retractions (\ref{eqOuterNoseRetraction}) and
(\ref{eqInnerNoseRetraction}) to decrease $b$.

Among the outer nose retractions (\ref{eqOuterNoseRetraction}),
cases (b) and (d) are problematic.
However, if $\arcs_n$ is large, $\arcs_n > 2\arcs/3$,
then case (d) can be applied $(n-1)$ times:
\[
\components(\arcs_1,\ldots,\arcs_n) \;=\;
\components(\arcs_1,\ldots,\arcs_{n-1},\arcs_n-2(\arcs-\arcs_n)),
\qquad \arcs_n > 2\arcs/3.
\]
Further iteration yields
\begin{equation}\label{eqIteratedOuterNoseRetractionPreD}
\components(\arcs_1,\ldots,\arcs_n) \;=\;
\components(\arcs_1,\arcs_2,\ldots,\arcs_{n-1},
  \remainder(\arcs_n,2(\arcs-\arcs_n)),
\end{equation}
again with the remainder $\remainder$ of the integer division.
Similarly, case (b) can be iterated, as long as its condition,
$\alpha_1 < \alpha_n < 2\alpha_1$, remains valid:
\begin{equation}\label{eqIteratedOuterNoseRetractionPreB}
\components(\arcs_1,\ldots,\arcs_n) \;=\;
\components(\remainder(\arcs_1, \alpha_n-\alpha_1),\arcs_2,\ldots,\arcs_{n-1},
  \alpha_n-\alpha_1+\remainder(\arcs_1,\alpha_n-\alpha_1)).
\end{equation}
Indeed, during the iteration,
the difference $\alpha_n-\alpha_1$ remains constant.

\begin{thm}[Outer nose retraction]\label{thmIteratedOuterNoseRetraction}
The outer nose retractions yield the algorithm
\begin{equation}\label{eqIteratedOuterNoseRetraction}
\begin{array}{l}
\components(\arcs_1,\arcs_2,\ldots,\arcs_{n-1},\arcs_n) \; = \\ \quad =
\left\{\begin{array}{lll}
\components(\arcs_n,\arcs_{n-1},\ldots,\arcs_2,\arcs_1),
  & \arcs_1 > \arcs_n, & (a)\\
\components(\arcs_2,\ldots,\arcs_{n-1}) + \arcs_1,
  & \arcs_1 = \arcs_n, & (b)\\
\components(\remainder(\arcs_1, \alpha_n\!-\!\alpha_1),
  \arcs_2,\ldots,\arcs_{n-1},
  \alpha_n\!-\!\alpha_1+\remainder(\arcs_1,\alpha_n\!-\!\alpha_1)),
  & \arcs_1 < \arcs_n < 2\arcs_1, & (c)\\
\components(\arcs_2,\ldots,\arcs_{n-1},\arcs_1),
  & 2\arcs_1 = \arcs_n, & (d)\\
\components(\arcs_2,\ldots,\arcs_{n-1},\arcs_1, \arcs_n-2\arcs_1),
  & 2\arcs_1 < \arcs_n < \frac{2}{3}\arcs, & (e)\\
\components(\arcs_1,\arcs_2,\ldots,\arcs_{n-1} ),
  & 2(a-\arcs_n) \,\Big|\, \arcs_n, & (f)\\
\components(\arcs_1,\arcs_2,\ldots,\arcs_{n-1},
  \remainder(\arcs_n, 2(\arcs-\arcs_n))),
  & otherwise, & (g)
\end{array}\right.
\end{array}\hspace*{-1em}
\end{equation}
with logarithmic complexity $\Ord(b)\Ord(n)$
to determine the number $\components$ of connected components of the
bi-rainbow meander $\rainbowMeander(\arcs_1,\ldots,\arcs_n)$.
\end{thm}

\begin{proof}
The validity of the algorithm follows
directly from (\ref{eqOuterNoseRetraction}) of
lemma \ref{thmOuterNoseRetraction} and the observations
(\ref{eqIteratedOuterNoseRetractionPreD},
\ref{eqIteratedOuterNoseRetractionPreB}) above.
Note the special case (f), which is case (g) with zero remainder.

Cases (b,c,d,f,g) reduce the number $b$ of bits.
Case (a) cannot be applied twice in succession,
in fact it could be replaced by symmetric copies of (c--g).
Finally, case (e) can be applied at most $(n-2)$ times in succession.
This yields the claimed complexity of the algorithm.
\end{proof}

Although we have found an algorithm of similar complexity than
the Euclidean algorithm, the number of cases is quite large.
The inner nose retraction (\ref{eqInnerNoseRetraction}) turns out to be
more beautiful.

\begin{thm}[Inner nose retraction]\label{thmIteratedInnerNoseRetraction}
The inner nose retraction yields the algorithm
\begin{equation}\label{eqIteratedInnerNoseRetraction}
\begin{array}{l}
\components(\arcs_1,\arcs_2,\ldots,\arcs_{n-1},\arcs_n) \; = \\ \quad =
\left\{\begin{array}{lll}
\components(\arcs_1,\ldots,\arcs_{m^*-1},\arcs_{m^*+1},\ldots,\arcs_n)
  + \arcs_{m^*},
  & L^* = R^*, & (a)\\
\components(\arcs_1,\ldots,\arcs_{m^*-1},\arcs_{m^*+1},\ldots,\arcs_n) ,
  & |L^*\!-\!R^*| \,\Big|\, \arcs_{m^*}, & (b)\\
\components(\arcs_1,\ldots,\arcs_{m^*-1},\remainder(\arcs_{m^*},|L^*\!-\!R^*|),
  \arcs_{m^*+1},\ldots,\arcs_n),
  & otherwise, & (c)
\end{array}\right.
\end{array}\hspace*{-1em}
\end{equation}
with logarithmic complexity $\Ord(b)\Ord(n)$
to determine the number $\components$ of connected components of the
bi-rainbow meander $\rainbowMeander(\arcs_1,\ldots,\arcs_n)$.
The values $m^*, L^*, R^*$ denote the index of the middle rainbow family and
the total numbers of arcs in the left and right rainbow families,
as defined in (\ref{eqMiddleRainbow}).
\end{thm}

\begin{proof}
The validity of the algorithm follows
again by iteration of (\ref{eqInnerNoseRetraction}) of
lemma \ref{thmInnerNoseRetraction}.
Indeed, as long as $\arcs_{m^*}$
after application of (\ref{eqInnerNoseRetraction})(c)
is not smaller than $|L^*-R^*|$,
the $m^*$-th family remains the middle one.
Furthermore, the values $L^*$ and $R^*$ do not change.
Iteration yields case (c) of (\ref{eqIteratedInnerNoseRetraction}) with
the special case (b) of zero remainder.

All cases of (\ref{eqIteratedInnerNoseRetraction}) reduce the number
$b$ of bits.
However, the values $m^*, L^*, R^*$ need to be computed in every step.
Together, we again find the bound $\Ord(b)\Ord(n)$ on the complexity
of the algorithm.
\end{proof}

For rainbow families $\arcs_k$ of similar size, the update of $m^*, L^*, R^*$
can be done starting from the old values.
The old and new midpoints should then only be $\Ord(1)$ apart.
This is similar to theorem \ref{thmIteratedOuterNoseRetraction} where
the factor $\Ord(n)$ is due to the case of $\arcs_n$ very large with respect to
all the other families.

The resulting complexity is therefore expected to be rather close to $\Ord(b)$
for ``typical'' bi-rainbow meanders.

The inner nose-retraction algorithm (\ref{eqIteratedInnerNoseRetraction})
very closely resembles the Euclidean algorithm (\ref{eqEuclideanAlgorithm}).
Indeed, the main operation of both algorithms is the remainder of an
integer division.
In the case of two upper families,
$\components(\arcs_1,\arcs_2) = \gcd(\arcs_1,\arcs_2)$,
both algorithms are in fact identical.


\section{Discussion \& outlook\label{secDiscussion}}

We have studied the existence of closed expressions
for the number of Jordan curves of a bi-rainbow meander.

On the one hand, there might be no convenient formula
in the greatest common divisor of the sizes of the involved
rainbow families.
Although theorem \ref{thmNoGcd} does not exclude all possible gcd formulae,
its proof shows major obstacles and confirms the respective conjecture 19 of
\cite{CollMagnantWang2012-Meander}.
The homogeneity assumption, for example,
can be weakened by a more careful scaling argument.
More arguments $f_k$ of the gcd can be allowed for meanders with more than
4 upper families. Indeed, linearity of $f_k$ follows inductively as in step 2
of the proof of theorem \ref{thmNoGcd}. The pigeonhole principle, step 4,
can be applied for bi-rainbows with at least $2^m$ upper families for
formulae of the form $\gcd(f_1,\ldots,f_m)$ with $m$ arguments.

On the other hand, instead of looking for more complicated gcd formulae,
we found an Euclidean-like algorithm to determine the number of Jordan curve.
Just as the Euclidean algorithm computed the gcd, suitably combined
nose retractions determine the number of Jordan curves in logarithmic time.
Moreover, the main step computes the remainder of an integer division
in close similarity to the Euclidean algorithm.
The number of Jordan curves of a bi-rainbow meander thus becomes
another number-theoretic quantity similar to the gcd.

So far, we have dealt with the special case of bi-rainbow meanders.
For most of the applications, this is only a first step.
In a forthcoming paper \cite{KarnauhovaLiebscher2015-GeneralNoseRetraction},
we shall describe logarithmic algorithms by nose retractions of
general meanders.

Note, however, that logarithmic algorithms require a representation of
the meander of logarithmic size.
Indeed, the algorithm has to at least read the input.
If the meander is represented as a product of transpositions or a permutation,
sections \ref{secMeanderTranspositions}--\ref{secMeanderShooting},
then no algorithm can be faster than $\Ord(\alpha)$~---
just as a direct inspection of the meander curves.
Indeed, traversing all $2\alpha$ arcs certainly provides
the number of closed curves.

The condensed bracket expression, section \ref{secMeanderBracket},
is therefore a prerequisite of the logarithmic algorithm.
Without it, the complexity advantage over a direct inspection is lost.
However, at least the structural properties of the nose retractions remain.

\clearpage

\bibliographystyle{alpha_abbrv}
\pdfbookmark{References}{secReferences}
\let\OLDthebibliography\thebibliography
\renewcommand\thebibliography[1]{
  \OLDthebibliography{CMW12} 
  }

\bibliography{KarLie-Meander}

\end{document}